\documentclass[11pt]{article}
\usepackage{amsmath, amsthm, amssymb, amsfonts, mathtools, bm}
\usepackage{graphicx}
\usepackage[hidelinks]{hyperref}
\usepackage[a4paper,margin=20mm]{geometry}
\usepackage{algorithm}
\usepackage{algorithmic}
\usepackage{enumitem}
\usepackage{autonum}
\usepackage{subcaption}

\makeatletter
\renewcommand\paragraph{\@startsection{paragraph}{4}{\z@}%
  {3.25ex \@plus1ex \@minus.2ex}%
  {-1em}%
  {\normalfont\normalsize\bfseries}}
\makeatother
\makeatletter
\let\origparagraph\paragraph
\renewcommand{\paragraph}[1]{%
  \origparagraph{#1.}%
}
\makeatother

\newtheorem{theorem}{Theorem}[section]
\newtheorem{lemma}[theorem]{Lemma}
\newtheorem{proposition}[theorem]{Proposition}

\theoremstyle{definition}

\newtheorem{remark}[theorem]{Remark}
\newtheorem{assumption}[theorem]{Assumption}

\makeatletter
\renewenvironment{proof}[1][\proofname]{\par
  \pushQED{\qed}%
  \normalfont
  \topsep6\p@\@plus6\p@\relax
  \trivlist
  \item[\hskip\labelsep\bfseries #1\@addpunct{.}]%
}{%
  \popQED\endtrivlist\@endpefalse
}
\makeatother

\newcommand{\R}{\mathbb{R}}
\newcommand{\E}{\mathbb{E}}
\newcommand{\norm}[1]{\left\lVert #1 \right\rVert}
\newcommand{\ip}[2]{\left\langle #1,\,#2 \right\rangle}
\newcommand{\dW}{\,\mathrm d W_t}
\newcommand{\dt}{\,\mathrm d t}

\newcommand{\di}{\,\mathrm d}

\title{Nonlocal Mean Field Schr\"{o}dinger Bridge with Learned Interactions}
\author{
  Daisuke Inoue\thanks{Department of Mathematics, Imperial College London, London SW7 2AZ, United Kingdom. \texttt{dinoue@ic.ac.uk}}
  \and
  Dante Kalise\thanks{Department of Mathematics, Imperial College London, London SW7 2AZ, United Kingdom. \texttt{d.kalise-balza@imperial.ac.uk}}
  \and
  Mathieu Lauri\`{e}re\thanks{Shanghai Frontiers Science Center of Artificial Intelligence and Deep Learning, NYU-ECNU Institute of Mathematical Sciences, NYU Shanghai, Shanghai 200126, People's Republic of China. \texttt{mathieu.lauriere@nyu.edu}}
}
\date{}

\begin{document}
\maketitle

\begin{abstract}
  The Schr\"odinger Bridge Problem connects an initial distribution to a terminal one along a minimum-energy stochastic process. Its mean-field extension, the Mean-Field Schr\"odinger Bridge, governs interacting populations whose dynamics and costs depend on the collective distribution. When these interactions are nonlocal, their direct evaluation scales
  quadratically with the population size, making large ensembles intractable
  within FBSDE-based solvers. We replace these terms with neural surrogates in
  state and time, trained on empirical interaction values along sampled
  trajectories and embedded in a four-stage alternating scheme that updates the forward and backward potentials and the surrogates in turn, while preserving forward--backward consistency and the prescribed endpoint marginals. We derive Gr\"onwall-type stability bounds quantifying how surrogate errors propagate to the generated trajectories under a small-gain condition. On crowd-navigation and high-dimensional opinion-dynamics benchmarks, the surrogates reproduce the trajectories obtained with exact evaluation at reduced training cost. The advantage is most significant when the interaction is a nonlinear functional of the measure, such as the normalized bounded-confidence drift, for which random-batch subsampling is biased and unstable whereas the learned surrogate remains accurate.
\end{abstract}

\section{Introduction}

The Schr\"odinger Bridge Problem (SBP) constructs a stochastic process that connects an initial distribution $\rho_0$ to a terminal distribution $\rho_T$ with minimum transport energy.
First posed by Schr\"odinger~\cite{Erwin1931Uber}, it is now read as an entropy-regularized optimal transport problem~\cite{Leonard2013surveya,Chen2021Stochastic}, a connection that underlies most current numerical algorithms.
The framework has been applied to stochastic optimal control~\cite{Chen2016Optimal,Chen2016Optimala,Chen2018Optimal,Caluya2022Wasserstein,Caluya2021Reflected} and probabilistic inference~\cite{Vargas2021Solving,Chen2021Likelihood}.
More recently, diffusion-based generative models have also cast trajectory generation between data and prior distributions as an SBP~\cite{DeBortoli2021Diffusiona,Liu2023I$^2$SB,Shi2023Diffusiona}.
A natural extension of SBP, and the subject of this paper, is to \emph{interacting} particle systems, in which a population of agents is transported under dynamics and costs that depend on the collective distribution.
Such problems appear in swarm robotics with collision avoidance~\cite{Zheng2022Transporting,Cui2023Scalable,Rapakoulias2025Steeringa}, crowd density steering~\cite{Caluya2022Wasserstein,Caluya2021Reflected}, and trajectory reconstruction from cellular movement data~\cite{Schiebinger2019OptimalTransport,Tong2020TrajectoryNet,Zhang2025Modeling}.
In the finite-particle formulation with $N$ agents, each agent's dynamics and cost depend on the configuration of the other $N-1$ agents.
Evaluating these pairwise interactions requires $O(N^2)$ operations per time step, making direct optimization intractable for large populations.

This motivates the mean-field limit~\cite{Sznitman1991Topics,Lacker2016General,Carmona2018Probabilistic}: as $N\to\infty$, a typical agent interacts with a continuum density $\rho(x,t)$ rather than with $N-1$ individual agents.
The $N$-particle problem then becomes a problem on the space of probability measures, the \emph{Mean-Field Schr\"odinger Bridge} (MFSB)~\cite{Backhoff2020Mean,Hernandez2025Propagation}.
Its optimality conditions are given by a coupled Hamilton--Jacobi--Bellman (HJB) and Fokker--Planck (FP) system~\cite{Pham2018Bellman}, which admits a probabilistic representation as Forward--Backward Stochastic Differential Equations (FBSDEs)~\cite{Chen2021Likelihood}.
Building on this representation, Liu et al.~\cite{Liu2022Deepa} proposed a deep learning method that solves the MFSB by alternately training forward and backward neural networks to satisfy the FBSDE system, bringing MFSBs within reach of dimensions out of scope for classical PDE solvers~\cite{Han2018Solving,Ruthotto2020Machine}.
Because the FBSDE solver trains on finite particle ensembles, its cost depends on how the distribution-dependent terms are evaluated along the simulated trajectories.
The \emph{local} and \emph{aggregated} distributional terms considered by Liu et al.~\cite{Liu2022Deepa} can be evaluated without forming all pairwise interactions: the former depend on the pointwise density $\rho(x,t)$, while the latter are batch-level statistics shared across particles, such as the mean $\E_{\rho_t}[y]$.
In contrast, the \emph{nonlocal} interactions considered here require a separate empirical population average at each evaluation point $x$; kernel convolutions $\int k(x,y)\rho(y,t)\,\mathrm{d}y$ are a representative example.
Such terms arise in pairwise collision avoidance and crowd congestion, but direct evaluation for all $N$ particles and all training iterations costs $O(N^2 K M_{\text{iter}})$, with $K$ the number of time steps and $M_{\text{iter}}$ the number of training iterations.

In this paper, we build on the FBSDE-based framework to address nonlocal interactions in the MFSB.
The idea is to replace the analytically known but expensive nonlocal terms in the dynamics and running cost by \emph{neural surrogates}.
During training, the exact empirical interaction value is evaluated on sampled particle trajectories and used as supervision; after training, the surrogate returns the corresponding value from the particle state and time alone.
A related line of work approximates known distribution-dependent interactions or controls by neural networks: deep solvers for McKean--Vlasov FBSDEs~\cite{Han2024Learning,Hu2025Deep}, neural population-dependent controls in mean-field control~\cite{Dayanikli2024Deep}, and random-Fourier-feature surrogates for known kernel interactions~\cite{Cao2026Scalable}.
Compared with these works, the present setting is the Schr\"odinger bridge, in which the surrogate must coexist with the forward--backward consistency and the prescribed endpoint marginals.
Concretely, we use separate networks for these two nonlocal terms and organize training as a four-stage loop that alternates between updating the forward/backward potentials and updating the surrogates.
The contributions of this paper are:
\begin{enumerate}[label=\arabic*)]
  \item A proposal of a four-stage alternating algorithm for the nonlocal MFSB.
        In addition to the forward and backward FBSDE potentials, the algorithm trains separate neural surrogates for the nonlocal term in the dynamics and the nonlocal term in the running cost, reducing the per-step evaluation cost from $O(N^2)$ to $O(N)$ in the trajectory inference steps.
  \item A stability analysis for the FBSDE system of the MFSB under approximation of the interaction terms.
        The analysis yields Gr\"onwall-type bounds that quantify how surrogate errors in the nonlocal dynamics and running cost propagate to the forward and backward FBSDE solutions, and to the combined Schr\"odinger bridge satisfying both marginal constraints.
  \item Numerical experiments on two navigation tasks and a high-dimensional opinion-dynamics task.
        The learned surrogates reproduce trajectories obtained with analytical evaluation while reducing training time across the tested problems.
\end{enumerate}

The rest of the paper is organized as follows.
Section~\ref{sec:problem} formulates the finite-particle control problem, passes to the mean-field Schr\"odinger bridge, and derives the HJB--FP and FBSDE characterizations that underlie the learning algorithm.
Section~\ref{sec:formulation} starts from the FBSDE losses proposed by Liu et al.~\cite{Liu2022Deepa}, then introduces neural surrogates for the nonlocal drift and cost and combines them into a four-stage alternating training procedure.
Section~\ref{sec:theory} analyzes how surrogate errors in the nonlocal drift and running cost affect the forward and backward stochastic systems.
Section~\ref{sec:experiments} evaluates the method on crowd-navigation and opinion-dynamics benchmarks, comparing analytical evaluation, the learned surrogate, and the random-batch approximation.
Section~\ref{sec:conclusion} summarizes the findings and discusses remaining theoretical questions.

\paragraph{Related Work}

The link between the SBP and optimal transport has been studied from a stochastic control perspective~\cite{Leonard2013surveya,Chen2021Stochastic,DaiPra1991stochastic}, with numerical methods including Sinkhorn-type algorithms~\cite{Cuturi2013Sinkhorn,Peyre2019Computational}, FBSDE-based approaches~\cite{Vargas2021Solving,Chen2021Likelihood}, and neural-network methods~\cite{DeBortoli2021Diffusiona,Shi2023Diffusiona,Gushchin2024Light}.
The mean-field extension was formulated by Backhoff et al.~\cite{Backhoff2020Mean}; Hern\'andez and Tangpi~\cite{Hernandez2025Propagation} subsequently established existence, optimality conditions, and propagation of chaos via terminal-marginal penalization.
For the same MFSB setting, Eldesoukey, Chen, and Halder~\cite{Eldesoukey2026Generalized} recently developed a generalized Sinkhorn algorithm for the associated nonlocal PDE system.
On the algorithmic side, Liu et al.~\cite{Liu2022Deepa} treat the MFSB as a mean-field control problem with state- and distribution-dependent running costs, including local or aggregated distributional terms whose particle evaluation does not create a pairwise bottleneck; the present work extends this FBSDE-based neural approach to nonlocal interactions, which are expensive to evaluate directly.
More recently, Nakano~\cite{Nakano2026Kernelbased} studied a subclass of potential mean-field games in which terminal constraints are relaxed and interaction costs are represented through kernel maximum mean discrepancy penalties, and developed unbiased random Fourier $U$-statistics to reduce the cost of kernel evaluations.
In contrast, we use learned surrogates for nonlocal interactions in the FBSDE drift and running cost, without assuming a kernel or random-feature structure.

Interaction learning in many-agent systems splits into two settings: recovering \emph{unknown} interaction laws from data, and approximating \emph{known} interactions with cheaper surrogates. Our work falls in the second.
In the first setting, trajectory-based nonparametric estimators with theoretical guarantees extend to heterogeneous and stochastic systems~\cite{Lu2019Nonparametric,Lu2021Learninga,Lu2022Learning}; complementary directions include likelihood-based drift estimation~\cite{Yao2022Meanfield}, density-based identification~\cite{Lang2022Learning,shiExtractinginteractionkernels2025}, and inverse problems for mean-field games that recover interaction costs from aggregate data~\cite{Yang2018Learninga,Ding2022Mean,Liu2023Inverse,Mo2024GameTheoretic}.
In the second, Dayanikli et al.~\cite{Dayanikli2024Deep} learn population-dependent controls with neural embeddings, Han et al.~\cite{Han2024Learning} and Hu et al.~\cite{Hu2025Deep} build deep solvers for McKean--Vlasov FBSDEs, and Cao et al.~\cite{Cao2026Scalable} accelerate kernel-based training via random Fourier features; broader kernel-based surrogate modeling is surveyed in~\cite{Fiedler2025Recent}.
Compared with these, we aim to accelerate the MFSB with known but expensive nonlocal interactions while preserving the forward--backward consistency and prescribed endpoint marginals of the Schr\"odinger bridge. To our knowledge, this is the first surrogate-acceleration strategy specifically for this setting.

\paragraph{Notation}
We denote by $\mathcal{P}(\mathbb{R}^d)$ the space of probability measures on $\mathbb{R}^d$ equipped with the 2-Wasserstein metric.
The notation $\mathbb{E}[\cdot]$ denotes expectation, $\nabla_x$ denotes the spatial gradient, and $\Delta_x$ denotes the spatial Laplacian.
The symbol $\|\cdot\|$ denotes the Euclidean norm for vectors and the Frobenius norm for matrices.

\section{Problem Formulation and FBSDE Characterization}\label{sec:problem}

This section starts from the finite-particle control problem, takes its mean-field limit, derives the associated HJB--FP equations, and recalls the FBSDE representation used by the learning algorithm.

\subsection{Finite Particle Control Problem}

Consider $N$ interacting particles in $\mathbb{R}^d$.
Let $(\Omega,\mathcal{F},(\mathcal{F}_t)_{t\in[0,T]},\mathbb{P})$ be a filtered probability space supporting $N$ independent $d$-dimensional Brownian motions $\{W_i\}_{i=1}^N$.
The particle states $x_i(t)\in\mathbb{R}^d$ are driven by $\mathcal{F}_t$-adapted controls $\alpha_i(t)\in\mathbb{R}^d$.
We consider the cost functional and dynamics
\begin{align}\label{eq:micro}
   & \inf_{\{\alpha_i\}_{i=1,\ldots,N}}\; \sum_{i=1}^N \E\!\left[\int_0^T \frac{1}{2\sigma^2}\|\alpha_i(t)\|^2 + F(x_i(t),t, \rho_t^N)\,\di t\right], \\
   & \quad \text{s.t.}\quad
  \di x_i(t) = \bigl(\alpha_i(t) + f(x_i(t),t, \rho_t^N)\bigr)\,\di t + \sigma\,\di W_i(t)
  ,                                                                                                                                                   \\
   & \qquad \qquad x_i(0)\sim\rho_0,\ x_i(T)\sim\rho_T.
\end{align}
Here $F:\R^d\times[0,T]\times\mathcal{P}(\R^d)\to\R$ is the nonlocal running cost, $f:\R^d\times[0,T]\times\mathcal{P}(\R^d)\to\R^d$ the nonlocal drift, $\sigma\in\R$ the diffusion coefficient, and
\[
  \rho_t^N \coloneqq \frac{1}{N}\sum_{j=1}^N \delta_{x_j(t)}
\]
is the empirical distribution of the particle system.

Unlike standard stochastic control, both the initial and terminal distributions of the particle laws are prescribed,
so classical stochastic-control techniques do not apply directly.
When $f=F=0$, it reduces to the classical Schr\"odinger Bridge, i.e., entropy-regularized optimal transport minimizing the relative entropy of the path measure with respect to the reference diffusion, subject to the two endpoint marginals.
We distinguish these terms from local or aggregated distributional terms: local terms depend on quantities such as $\rho(x,t)$ at the evaluation point, while aggregated terms are population-level statistics that can be computed once and shared across particles.
This paper focuses on \emph{nonlocal} interactions, where the drift or cost depends on a population average evaluated separately at each particle state.
A canonical example is a spatial convolution of the density:
\begin{align}
  \begin{aligned}\label{eq:nonlocal-form}
    f(x_i(t),t,\rho_t^N) & = \int k_f\bigl(x_i(t),x\bigr) \rho_t^N(\di x) = \frac{1}{N}\sum_{j=1}^N k_f\bigl(x_i(t),x_j(t)\bigr) , \\
    F(x_i(t),t,\rho_t^N) & = \int k_F\bigl(x_i(t),x\bigr) \rho_t^N(\di x) = \frac{1}{N}\sum_{j=1}^N k_F\bigl(x_i(t),x_j(t)\bigr) ,
  \end{aligned}
\end{align}
for kernels $k_f:\R^d\times\R^d\to\R^d$ and $k_F:\R^d\times\R^d\to\R$.
A naive evaluation of such nonlocal interactions costs $O(N^2)$ pairwise operations per time step, which makes the finite-$N$ problem intractable for large populations.

\subsection{Mean-Field Schr\"odinger Bridge}

At the formal level, as $N\to\infty$, the empirical distribution $\rho_t^N$ is approximated by a deterministic density $\rho(\cdot,t)$.
A rigorous justification of this limit for Schr\"odinger-type control problems, including propagation of chaos for the underlying particle system, is given in~\cite{Hernandez2025Propagation}.
The Mean-Field Schr\"odinger Bridge (MFSB) is the resulting minimization of the expected cost of a representative particle subject to the transport constraint:
\begin{align}
   & \min_{\alpha} \int_0^T\!\int \Big[\frac{1}{2\sigma^2}\|\alpha(x,t)\|^2 + F(x,t, \rho_t)\Big]\rho(x,t)\,\di x\,\di t, \label{eq:mfsp} \\
   & \quad \text{s.t.}\quad
  \partial_t \rho(x,t) = - \nabla_x\cdot\bigl(\rho(x,t)\, (\alpha(x,t) + f(x,t, \rho_t))\bigr) + \frac{\sigma^2}{2}\Delta_x \rho(x,t),    \\
   & \qquad \quad \rho(\cdot,0)=\rho_0,\ \rho(\cdot,T)=\rho_T. \label{eq:fp-rho}
\end{align}
With $\rho_t=\rho(\cdot,t)$, the interaction terms take the form
\begin{align}
  f(x,t,\rho_t) & = \int k_f(x,y)\,\rho(y,t)\,\di y, \\
  F(x,t,\rho_t) & = \int k_F(x,y)\,\rho(y,t)\,\di y.
\end{align}

\subsection{Coupled PDEs corresponding to the MFSB}

We formally apply Pontryagin's maximum principle to \eqref{eq:mfsp}--\eqref{eq:fp-rho};
a rigorous derivation via penalization of the terminal cost is given in \cite{Hernandez2025Propagation}.
Define the Hamiltonian
\begin{align}
  H(x,\rho,\nabla_x u,t)
  \coloneqq
  \sup_{\alpha}\Big\{
  -\alpha^\top \nabla_x u
  - \frac{1}{2\sigma^2}\|\alpha\|^2
  \Big\}
  - \nabla_x u^\top f
  - F .
\end{align}
The first-order optimality condition gives
$\alpha^\ast = -\sigma^2 \nabla_x u$, so that
\begin{align}
  H
  =
  \frac{\sigma^2}{2}\|\nabla_x u\|^2
  - \nabla_x u^\top f
  - F .
\end{align}
The resulting optimality conditions take the form of coupled Hamilton--Jacobi--Bellman (HJB) and Fokker--Planck (FP) equations
\begin{align}
  -\partial_t u
   & =
  -\frac{\sigma^2}{2}\|\nabla_x u\|^2
  + \nabla_x u^\top f
  + \frac{\sigma^2}{2}\Delta_x u
  + F, \label{eq:hjb-new} \\
  \partial_t \rho
   & =
  \nabla_x\cdot\bigl(\rho(\sigma^2\nabla_x u - f)\bigr)
  + \frac{\sigma^2}{2}\Delta_x\rho, \label{eq:fp-new}
\end{align}
subject to the boundary conditions
$\rho(\cdot,0)=\rho_0$ and $\rho(\cdot,T)=\rho_T$.

We linearize \eqref{eq:hjb-new}--\eqref{eq:fp-new} via the Cole--Hopf transformation~\cite{Leger2021Hopf,DaiPra1991stochastic}
\begin{equation}\label{eq:cole-hopf}
  \Psi(x,t) \coloneqq \exp\bigl(-u(x,t)\bigr),
  \qquad
  \widehat\Psi(x,t) \coloneqq \rho(x,t)\exp\bigl(u(x,t)\bigr),
\end{equation}
so that $\rho = \Psi\widehat\Psi$.
Substituting into \eqref{eq:hjb-new}--\eqref{eq:fp-new} and using standard identities for the Laplacian of $\log\Psi$ yields the linear forward--backward system
\begin{align}
  \partial_t \Psi
   & =
  -\nabla_x\Psi^\top f
  - \frac{\sigma^2}{2}\Delta_x\Psi
  + F\Psi, \label{eq:psi} \\
  \partial_t \widehat\Psi
   & =
  -\nabla_x\cdot(\widehat\Psi f)
  - \frac{\sigma^2}{2}\Delta_x\widehat\Psi
  - F\widehat\Psi. \label{eq:psi-dagger}
\end{align}
The boundary conditions on $\rho$ become
\begin{equation}\label{eq:bc-psi}
  \Psi(\cdot,0)\widehat\Psi(\cdot,0)=\rho_0,
  \qquad
  \Psi(\cdot,T)\widehat\Psi(\cdot,T)=\rho_T.
\end{equation}
Equations \eqref{eq:psi}--\eqref{eq:psi-dagger} are the forward/backward Schr\"odinger system associated with the MFSB.

\subsection{FBSDEs corresponding to the MFSB}

Liu et al.~\cite{Liu2022Deepa} derived a probabilistic representation for \eqref{eq:psi}--\eqref{eq:psi-dagger} via It\^o's formula.
Define the logarithmic value functions
\begin{align}
  Y(x,t) \coloneqq -\log \Psi(x,t),\qquad Z(x,t) \coloneqq \sigma\nabla_x Y(x,t), \\
  \widehat Y(x,t) \coloneqq -\log \widehat\Psi(x,t),\qquad \widehat Z(x,t) \coloneqq \sigma\nabla_x \widehat Y(x,t).
\end{align}
Applying It\^o's formula to $Y(X_t, t)$ and $\widehat Y(X_t, t)$ along the diffusion processes gives the FBSDE system below.

The forward-time FBSDE for \eqref{eq:psi} reads
\begin{subequations}\label{eq:fbsde-forward-full}
  \begin{align}
    \di X_t          & = \big(f(X_t,t,\rho_t)+\sigma Z_t\big)\dt + \sigma\,\di W_t,\label{eq:fbsde-forward-full-x}                                                                                                                                      \\
    \di Y_t          & = \Big(\tfrac12\norm{Z_t}^2 + F(X_t,t,\rho_t)\Big)\dt + Z_t^\top \dW,\label{eq:fbsde-forward-full-y}                                                                                                                             \\
    \di \widehat Y_t & = \Big(\tfrac12\norm{\widehat Z_t}^2 + \nabla_x\!\cdot\big(\sigma \widehat Z(\cdot,t) - f(\cdot,t,\rho_t)\big)(X_t) + \widehat Z_t^\top Z_t - F(X_t,t,\rho_t)\Big)\dt + \widehat Z_t^\top \dW.\label{eq:fbsde-forward-full-yhat}
  \end{align}
\end{subequations}
The divergence term $\nabla_x\cdot(\sigma\widehat Z - f)$ in \eqref{eq:fbsde-forward-full-yhat} comes from the second-order term in It\^o's formula when converting the Fokker--Planck-type PDE \eqref{eq:psi-dagger} into a backward SDE for the log-density.

The backward-time FBSDE, solved from $T$ to $0$, is
\begin{subequations}\label{eq:fbsde-backward-full}
  \begin{align}
    \di \widehat X_t
     & = \big(-f(\widehat X_t,t,\widehat\rho_t)+\sigma \widehat Z_t\big)\dt
    + \sigma\,\di W_t,\label{eq:fbsde-backward-full-x}                                                                    \\
    \di \widehat{Y}_t
     & = \Big(\tfrac12\norm{\widehat{Z}_t}^2 + F(\widehat X_t,t,\widehat\rho_t)\Big)\dt
    + \widehat{Z}_t^\top \di W_t,\label{eq:fbsde-backward-full-yhat}                                                      \\
    \di Y_t
     & = \Big(\tfrac12\norm{Z_t}^2 + \nabla_x\!\cdot\big(\sigma Z(\cdot,t) + f(\cdot,t,\widehat\rho_t)\big)(\widehat X_t)
    + Z_t^\top \widehat{Z}_t - F(\widehat X_t,t,\widehat\rho_t)\Big)\dt
    + Z_t^\top \di W_t.\label{eq:fbsde-backward-full-y}
  \end{align}
\end{subequations}
Here $\widehat X_T\sim\rho_T$, $\widehat\rho_t$ denotes the distribution of $\widehat X_t$.
The divergence term $\nabla_x\cdot(\sigma Z + f)$ in \eqref{eq:fbsde-backward-full-y} arises from the dual relation.
A solution of the Schr\"odinger bridge problem is characterized by consistency between the forward process $(X_t)$ and the backward process $(\widehat X_t)$:
they induce the same path measure and satisfy $\rho_t = \widehat\rho_t$ for all $t\in[0,T]$.

The variables $Y$ and $\widehat Y$ are not independent value functions; they are complementary potentials whose sum recovers the density via the Cole--Hopf relation $\rho=\exp(-(Y+\widehat Y))$.
In the forward-time FBSDE, $Y$ is the HJB value function enforcing optimality of the controlled dynamics, while $\widehat Y$ is an auxiliary potential that aligns the density with the terminal constraint $\rho_T$.
In the backward-time FBSDE the roles are reversed: $\widehat Y$ is the backward HJB value function and $Y$ enforces consistency with the initial distribution $\rho_0$.
The two FBSDEs are thus two stochastic representations of the same pair $(\Psi,\widehat\Psi)$; their consistency characterizes the Schr\"odinger bridge under both marginal constraints.

\section{Learned Surrogates and a Four-Stage FBSDE Scheme}\label{sec:formulation}

Our method builds on the FBSDE learning framework of Liu et al.~\cite{Liu2022Deepa}.
The framework approximates the FBSDE solutions with neural networks.
For local or aggregated distributional terms, the required particle-batch evaluation is not the dominant cost.
For nonlocal interactions, however, the same evaluation becomes an empirical population average over all particles for each evaluation point.
To reduce this cost, we introduce learned surrogates for the nonlocal interaction terms.
We then analyze the resulting computational cost.

\subsection{Baseline FBSDE Learning Objectives}

Following~\cite{Liu2022Deepa}, we parameterize the unknown functions in the FBSDEs \eqref{eq:fbsde-forward-full}--\eqref{eq:fbsde-backward-full} by neural networks and train them with the losses below.
Our method keeps these losses and changes only how the interaction terms are evaluated.
For the forward SDE, $Y(x,t)$ and $Z(x,t)$ are represented by networks $\tilde Y_\theta(x,t)$ and $\tilde Z_\theta(x,t)$ with parameters $\theta$.
For the backward SDE, $\widehat Y(x,t)$ and $\widehat Z(x,t)$ are represented by $\tilde{\widehat Y}_\phi(x,t)$ and $\tilde{\widehat Z}_\phi(x,t)$ with parameters $\phi$.
On a time grid $0=t_0<t_1<\dots<t_K=T$, with $\Delta W_k=W_{t_{k+1}}-W_{t_k}$, samples from the forward SDE and the backward SDE are generated by the Euler--Maruyama updates
\begin{align}
  X_{t_{k+1}}
   & =
  X_{t_k}
  + f(X_{t_k}, t_k, \rho_{t_k})\Delta t
  + \sigma \tilde Z_\theta(X_{t_k},t_k)\Delta t
  + \sigma \Delta W_k, \\
  \widehat X_{t_k}
   & =
  \widehat X_{t_{k+1}}
  -\big(-f(\widehat X_{t_{k+1}}, t_{k+1}, \widehat\rho_{t_{k+1}})
  + \sigma \tilde{\widehat Z}_\phi(\widehat X_{t_{k+1}},t_{k+1})\big)\Delta t
  - \sigma \Delta W_k.
\end{align}
The losses below are evaluated on the trajectory used in the corresponding update in Algorithm~\ref{alg:fourstep_alt_resample}: backward losses for $\phi$ use the forward trajectory $X$, while forward losses for $\theta$ use the backward trajectory $\widehat X$.

\paragraph{IPF Loss}
The Iterative Proportional Fitting (IPF) loss is
\begin{align}
  \mathcal{L}_{\mathrm{IPF}}^{\mathrm{fwd}}(\theta)
   & \coloneqq \sum_{k=0}^{K-1} \Delta t\, \E\!\left[
                                                 \frac{1}{2}\norm{\tilde Z_\theta(\widehat X_{t_k},t_k) + \tilde{\widehat Z}_\phi(\widehat X_{t_k},t_k)}^2
                                                 + \nabla_x\!\cdot\big(\sigma \tilde Z_\theta(\cdot,t_k) - f(\cdot,t_k,\widehat\rho_{t_k})\big)(\widehat X_{t_k})
                                                 \right], \\
  \mathcal{L}_{\mathrm{IPF}}^{\mathrm{bwd}}(\phi)
   & \coloneqq \sum_{k=0}^{K-1} \Delta t\, \E\!\left[
                                                 \frac{1}{2}\norm{\tilde{\widehat Z}_\phi(X_{t_k},t_k) + \tilde Z_\theta(X_{t_k},t_k)}^2
                                                 + \nabla_x\!\cdot\big(\sigma \tilde{\widehat Z}_\phi(\cdot,t_k) + f(\cdot,t_k,\rho_{t_k})\big)(X_{t_k})
                                                 \right].
\end{align}
These correspond to $\mathrm{KL}(q_\theta\Vert q_\phi)$ and $\mathrm{KL}(q_\phi\Vert q_\theta)$ for the forward and backward directions, and drive the two drifts toward each other.

\paragraph{TD Loss}
The Temporal Difference (TD) loss enforces temporal consistency of the value functions.
Discretizing \eqref{eq:fbsde-forward-full-yhat} along the forward trajectory, we compare $\tilde{\widehat Y}_\phi(X_{t_{k+1}},t_{k+1})$ to the target propagated from $t_k$:
\begin{align}
  \tilde{\widehat Y}_\phi(X_{t_{k+1}},t_{k+1})
  \approx
  \mathrm{TD}_{\phi|\theta}^{\mathrm{bwd}}(X_{t_k},t_k),
\end{align}
where
\begin{align}
  \mathrm{TD}_{\phi|\theta}^{\mathrm{bwd}} &
  (X_{t_k},t_k)
  \coloneqq\;
  \tilde{\widehat Y}_\phi(X_{t_k},t_k) \notag                                                                    \\
                                           & + \Delta t\Big(
  \tfrac12\norm{\tilde{\widehat Z}_\phi(X_{t_k},t_k)}^2
  + \nabla_x\!\cdot\big(\sigma \tilde{\widehat Z}_\phi(\cdot,t_k) - f(\cdot,t_k,\rho_{t_k})\big)(X_{t_k}) \notag \\
                                           & \hspace{2.6em}
  + \tilde{\widehat Z}_\phi(X_{t_k},t_k)^\top \tilde Z_\theta(X_{t_k},t_k)
  - F(X_{t_k},t_k,\rho_{t_k})
  \Big) \notag                                                                                                   \\
                                           & + \tilde{\widehat Z}_\phi(X_{t_k},t_k)^\top \Delta W_k .
\end{align}
At $t_K=T$, we impose the boundary condition
\begin{equation}\label{eq:td-terminal-bwd}
  \tilde{\widehat Y}_\phi(X_{t_K},t_K)
  \approx
  -\log \rho_T(X_{t_K}),
\end{equation}
which is the terminal constraint of the Schr\"odinger system.
The backward TD loss is then
\begin{align}\label{eq:td-loss-bwd}
  \mathcal{L}_{\mathrm{TD}}^{\text{bwd}}(\phi)
  \coloneqq
  \sum_{k=0}^{K-1}
  \E\!\left[
        \bigl(
        \tilde{\widehat Y}_\phi(X_{t_{k+1}},t_{k+1})
        -
        \mathrm{TD}_{\phi|\theta}^{\mathrm{bwd}}(X_{t_k},t_k)
        \bigr)^2
        \right]
  +
  \E\!\left[
        \bigl(
        \tilde{\widehat Y}_\phi(X_{t_K},t_K)
        + \log \rho_T(X_{t_K})
        \bigr)^2
        \right].
\end{align}
Analogously, discretizing \eqref{eq:fbsde-backward-full-y} backward from $t_{k+1}$ to $t_k$ gives the forward TD residual for interior steps:
\begin{align}\label{eq:td-loss-fwd}
  \mathcal{L}_{\mathrm{TD}}^{\text{fwd}}(\theta)
  \coloneqq
  \sum_{k=0}^{K-1}
  \E\!\left[
        \bigl(
        \tilde Y_\theta(\widehat X_{t_k},t_k)
        -
        \mathrm{TD}_{\theta|\phi}^{\mathrm{fwd}}(\widehat X_{t_{k+1}},t_{k+1})
        \bigr)^2
        \right]
  +
  \E\!\left[
        \bigl(
        \tilde Y_\theta(\widehat X_{t_0},t_0)
        + \log \rho_0(\widehat X_{t_0})
        \bigr)^2
        \right].
\end{align}
where
\begin{align}
  \mathrm{TD}_{\theta|\phi}^{\mathrm{fwd}}
   & (\widehat X_{t_{k+1}},t_{k+1})
  \coloneqq
  \tilde Y_\theta(\widehat X_{t_{k+1}},t_{k+1}) \notag                        \\
   & - \Delta t\Big(
  \tfrac12\norm{\tilde Z_\theta(\widehat X_{t_{k+1}},t_{k+1})}^2
  + \nabla_x\!\cdot\big(\sigma \tilde Z_\theta(\cdot,t_{k+1})
  + f(\cdot,t_{k+1},\widehat\rho_{t_{k+1}})\big)(\widehat X_{t_{k+1}}) \notag \\
   & \hspace{2.6em}
  + \tilde Z_\theta(\widehat X_{t_{k+1}},t_{k+1})^\top
  \tilde{\widehat Z}_\phi(\widehat X_{t_{k+1}},t_{k+1})
  - F(\widehat X_{t_{k+1}},t_{k+1},\widehat\rho_{t_{k+1}})
  \Big) \notag                                                                \\
   & - \tilde Z_\theta(\widehat X_{t_{k+1}},t_{k+1})^\top \Delta W_k .
\end{align}

\paragraph{FK Loss} The Feynman--Kac (FK) loss penalizes the gradient consistency $\sigma\nabla_x Y = Z$:
\begin{align}
  \mathcal{L}_{\mathrm{FK}}^{\mathrm{fwd}}(\theta)
   & \coloneqq \sum_{k=0}^{K} \E\!\left[
                                    \norm{\sigma\nabla_x \tilde Y_\theta(\widehat X_{t_k},t_k) - \tilde Z_\theta(\widehat X_{t_k},t_k)}^2
                                    \right],\label{eq:fk-loss-fwd} \\
  \mathcal{L}_{\mathrm{FK}}^{\mathrm{bwd}}(\phi)
   & \coloneqq \sum_{k=0}^{K} \E\!\left[
                                    \norm{\sigma\nabla_x \tilde{\widehat Y}_\phi(X_{t_k},t_k) - \tilde{\widehat Z}_\phi(X_{t_k},t_k)}^2
                                    \right].\label{eq:fk-loss-bwd}
\end{align}

The total learning objective is
\begin{equation}\label{eq:total-loss}
  \mathcal{L}_{\mathrm{tot}}(\theta,\phi)
  \coloneqq \lambda_{\mathrm{IPF}}\big(\mathcal{L}_{\mathrm{IPF}}^{\mathrm{fwd}}(\theta)+\mathcal{L}_{\mathrm{IPF}}^{\mathrm{bwd}}(\phi)\big)
  + \lambda_{\mathrm{TD}}\bigl(\mathcal{L}_{\mathrm{TD}}^{\mathrm{fwd}}(\theta) + \mathcal{L}_{\mathrm{TD}}^{\mathrm{bwd}}(\phi)\bigr)
  + \lambda_{\mathrm{FK}}\big(\mathcal{L}_{\mathrm{FK}}^{\mathrm{fwd}}(\theta)+\mathcal{L}_{\mathrm{FK}}^{\mathrm{bwd}}(\phi)\big),
\end{equation}
with $\lambda_\bullet$ the relative weights.

\subsection{Four-Stage Algorithm with Interaction Surrogates}\label{sec:algorithm}

The particle counterparts of the baseline losses require evaluating empirical interactions along sampled trajectories.
For large particle systems this is the dominant cost.
To remove this bottleneck, we train separate networks $\tilde f_\psi,\tilde F_\psi$ to match the empirical interaction values, in addition to the forward/backward potentials represented by $\tilde Y_\theta,\tilde Z_\theta,\tilde{\widehat Y}_\phi,\tilde{\widehat Z}_\phi$.
We write $X^{(i)}_{0:K}\coloneqq (X^{(i)}_{t_k})_{k=0}^K$ for one trajectory of particle $i$, and $\mathbf{X}_{0:K}\coloneqq \bigl(X^{(i)}_{0:K}\bigr)_{i=1}^{N}$ for one $N$-particle trajectory.
Backward trajectories are denoted analogously by $\widehat X^{(i)}_{0:K}$ and $\widehat{\mathbf X}_{0:K}$.
For such batches, the empirical laws are
\begin{align}
  \rho_{t_k}^N
  \coloneqq \frac1N\sum_{i=1}^N \delta_{X_{t_k}^{(i)}}, \quad
  \widehat\rho_{t_k}^N
  \coloneqq \frac1N\sum_{i=1}^N \delta_{\widehat X_{t_k}^{(i)}}.
\end{align}
For kernel interactions of the form \eqref{eq:nonlocal-form}, direct evaluation with $\rho_{t_k}^N$ uses the pairwise sums
\begin{align}\label{eq:empirical-pairwise-interactions}
  f(X_{t_k}^{(i)},t_k,\rho_{t_k}^N)
   & =
  \frac{1}{N}\sum_{j=1}^N k_f\bigl(X_{t_k}^{(i)},X_{t_k}^{(j)}\bigr), \\
  F(X_{t_k}^{(i)},t_k,\rho_{t_k}^N)
   & =
  \frac{1}{N}\sum_{j=1}^N k_F\bigl(X_{t_k}^{(i)},X_{t_k}^{(j)}\bigr).
\end{align}
These sums take $O(N^2)$ operations per time step.

On the particle trajectories generated during the forward/backward updates, we train $\tilde f_\psi,\tilde F_\psi$ with the empirical analytical values computed from these sums as supervision targets.
The trained surrogates $\tilde f_\psi(x,t),\tilde F_\psi(x,t)$ are then evaluated in $O(N)$ time in the drift-update steps, replacing the $O(N^2)$ analytical sums there.
The empirical interaction losses on forward and backward particle trajectories are
\begin{align}
  \mathcal{L}_{\mathrm{int}}^{N,\mathrm{fwd}}(\psi)
   & \coloneqq \sum_{k=0}^{K-1} \E\!\left[
                                      \frac1N\sum_{i=1}^N
                                      \norm{\tilde f_\psi(X_{t_k}^{(i)},t_k) - f(X_{t_k}^{(i)},t_k,\rho_{t_k}^N)}^2
                                      \right. \nonumber \\
                                      & \hspace{8em} \left.
                                      + \frac1N\sum_{i=1}^N
                                      \left|\tilde F_\psi(X_{t_k}^{(i)},t_k) - F(X_{t_k}^{(i)},t_k,\rho_{t_k}^N)\right|^2
                                      \right], \label{eq:int-loss-fwd}                           \\[4pt]
  \mathcal{L}_{\mathrm{int}}^{N,\mathrm{bwd}}(\psi)
   & \coloneqq \sum_{k=0}^{K-1} \E\!\left[
                                      \frac1N\sum_{i=1}^N
                                      \norm{\tilde f_\psi(\widehat X_{t_k}^{(i)},t_k) - f(\widehat X_{t_k}^{(i)},t_k,\widehat\rho_{t_k}^N)}^2
                                      \right. \nonumber \\
                                      & \hspace{8em} \left.
                                      + \frac1N\sum_{i=1}^N
                                      \left|\tilde F_\psi(\widehat X_{t_k}^{(i)},t_k) - F(\widehat X_{t_k}^{(i)},t_k,\widehat\rho_{t_k}^N)\right|^2
                                      \right]. \label{eq:int-loss-bwd}
\end{align}

Algorithm~\ref{alg:fourstep_alt_resample} summarizes the resulting four-stage training scheme, in which each outer iteration alternates between interaction updates and drift updates.
In the interaction updates, the surrogate networks $(\tilde f_\psi,\tilde F_\psi)$ are trained with the empirical interaction losses \eqref{eq:int-loss-fwd}--\eqref{eq:int-loss-bwd}, using the analytical pairwise sums evaluated on particle batches through the empirical laws $\rho_{t_k}^N$ and $\widehat\rho_{t_k}^N$ as supervision targets.
In the drift updates, the FBSDE networks $(\tilde Y_\theta,\tilde Z_\theta,\tilde{\widehat Y}_\phi,\tilde{\widehat Z}_\phi)$ are trained with the particle-average loss $\mathcal{L}_{\mathrm{tot}}^N$, the finite-particle counterpart of \eqref{eq:total-loss}, with the learned surrogates supplying the interaction terms.
The algorithm thus minimizes the combined empirical objective
\begin{align}\label{eq:augmented-loss}
  \min_{\theta,\phi,\psi}\ \mathcal{J}^N(\theta,\phi,\psi)
  \coloneqq \  & \mathcal{L}_{\mathrm{tot}}^N(\theta,\phi)+ \lambda_{\mathrm{int}}\mathcal{L}_{\mathrm{int}}^N(\psi),
\end{align}
where $\mathcal{L}_{\mathrm{int}}^N\coloneqq\mathcal{L}_{\mathrm{int}}^{N,\mathrm{fwd}}+\mathcal{L}_{\mathrm{int}}^{N,\mathrm{bwd}}$ and $\lambda_{\mathrm{int}}$ weights the interaction losses.

\begin{algorithm}[h]
  \caption{Four-Stage Training of MFSB}\label{alg:fourstep_alt_resample}
  \begin{algorithmic}[1]
    \WHILE{Not Converged}
    \STATE \textbf{Backward drift update ($\phi$):}
    \FOR{$j=1,\dots,M_{\mathrm{SDE}}$}
    \STATE Sample one $N$-particle batch $\mathbf{X}_{0}$ with $X_0^{(i)}\sim\rho_0$.
    \STATE Generate the forward trajectory $\mathbf{X}_{0:K},\{\Delta W_k\}$ using $(\tilde Y_\theta,\tilde Z_\theta,\tilde f_\psi,\tilde F_\psi)$.
    \STATE Update $(\tilde{\widehat Y}_\phi,\tilde{\widehat Z}_\phi)$ once on this trajectory using the backward part of the particle loss corresponding to \eqref{eq:total-loss}.
    \ENDFOR
    \STATE \textbf{Forward interaction update ($\psi$):}
    \FOR{$m=1,\dots,M_{\mathrm{int}}$}
    \STATE Sample one $N$-particle forward trajectory $\mathbf{X}_{0:K}$, compute $f(X_{t_k}^{(i)},t_k,\rho_{t_k}^N)$ and $F(X_{t_k}^{(i)},t_k,\rho_{t_k}^N)$ for all $i,k$ using \eqref{eq:empirical-pairwise-interactions}.
    \STATE Update $(\tilde f_\psi,\tilde F_\psi)$ once on the fixed trajectory $\mathbf{X}_{0:K}$ using $\mathcal{L}_{\mathrm{int}}^{N,\mathrm{fwd}}$ in \eqref{eq:int-loss-fwd}.
    \ENDFOR
    \STATE \textbf{Forward drift update ($\theta$):}
    \FOR{$j=1,\dots,M_{\mathrm{SDE}}$}
    \STATE Sample one $N$-particle batch $\mathbf{X}_{T}$ with $X_T^{(i)}\sim\rho_T$.
    \STATE Generate the backward trajectory $\widehat{\mathbf X}_{0:K},\{\Delta W_k\}$ using $(\tilde{\widehat Y}_\phi,\tilde{\widehat Z}_\phi,\tilde f_\psi,\tilde F_\psi)$.
    \STATE Update $(\tilde Y_\theta,\tilde Z_\theta)$ once on this trajectory using the forward part of the particle loss corresponding to \eqref{eq:total-loss}.
    \ENDFOR
    \STATE \textbf{Backward interaction update ($\psi$):}
    \FOR{$m=1,\dots,M_{\mathrm{int}}$}
    \STATE Sample one $N$-particle backward trajectory $\widehat{\mathbf X}_{0:K}$, compute $f(\widehat X_{t_k}^{(i)},t_k,\widehat\rho_{t_k}^N)$ and $F(\widehat X_{t_k}^{(i)},t_k,\widehat\rho_{t_k}^N)$ for all $i,k$ using the analogue of \eqref{eq:empirical-pairwise-interactions}.
    \STATE Update $(\tilde f_\psi,\tilde F_\psi)$ once on the fixed trajectory $\widehat{\mathbf X}_{0:K}$ using $\mathcal{L}_{\mathrm{int}}^{N,\mathrm{bwd}}$ in \eqref{eq:int-loss-bwd}.
    \ENDFOR
    \ENDWHILE
  \end{algorithmic}
\end{algorithm}

\subsection{Complexity Analysis}\label{sec:complexity}

We compare Algorithm~\ref{alg:fourstep_alt_resample} to a baseline that evaluates $f$ and $F$ analytically at every drift update.
We count the cost of one directional pass, since the forward and backward passes are symmetric.
Let $M_{\mathrm{SDE}}$ and $M_{\mathrm{int}}$ denote the numbers of drift-update and interaction-update steps in one such pass.
Let $C_{\mathrm{int}}^{\mathrm{ana}}$ be the cost of one analytical evaluation of the nonlocal interaction terms along an $N$-particle trajectory.
Let $C_{\mathrm{traj}}^{\mathrm{NN}}$, $C_{\mathrm{pol}}^{\mathrm{NN}}$, and $C_{\mathrm{int}}^{\mathrm{NN}}$ be, respectively, the neural costs of generating one trajectory, updating the forward/backward bridge networks, and performing one supervised interaction-network update.

In Algorithm~\ref{alg:fourstep_alt_resample}, each drift update generates a trajectory and updates the FBSDE networks using the learned surrogates.
Each interaction update generates a trajectory, computes analytical labels, and updates the interaction network.
The proposed cost per directional pass is therefore
\begin{align}
  {C}_{\mathrm{prop}}
   & =
  M_{\mathrm{SDE}}(C_{\mathrm{traj}}^{\mathrm{NN}}+C_{\mathrm{pol}}^{\mathrm{NN}})
  +M_{\mathrm{int}}(C_{\mathrm{traj}}^{\mathrm{NN}}+C_{\mathrm{int}}^{\mathrm{ana}}+C_{\mathrm{int}}^{\mathrm{NN}}).
\end{align}
The direct-computation baseline performs $M_{\mathrm{SDE}}$ drift updates with analytical evaluation of $f$ and $F$.
Each step generates a trajectory and updates the FBSDE networks as above, but also evaluates the nonlocal terms analytically during the drift update.
Its cost per directional pass is
\begin{equation}
  {C}_{\mathrm{direct}}
  =
  M_{\mathrm{SDE}}(C_{\mathrm{traj}}^{\mathrm{NN}}+C_{\mathrm{pol}}^{\mathrm{NN}}+C_{\mathrm{int}}^{\mathrm{ana}}).
\end{equation}
We say that Algorithm~\ref{alg:fourstep_alt_resample} is \emph{efficient relative to the direct-computation baseline} if it has lower cost for one directional pass, namely
\begin{equation}\label{eq:efficiency-condition}
  C_{\mathrm{prop}}<C_{\mathrm{direct}}.
\end{equation}

\begin{proposition}[Efficiency condition]\label{prop:efficiency-condition}
  Under the cost model above, Algorithm~\ref{alg:fourstep_alt_resample} is efficient relative to the direct-computation baseline if and only if
  \begin{equation}\label{eq:efficiency-algebraic-condition}
    \frac{M_{\mathrm{int}}}{M_{\mathrm{SDE}}}<\frac{1}{1+\frac{C^{\mathrm{NN}}}{C_{\mathrm{int}}^{\mathrm{ana}}}},
  \end{equation}
  where $C^{\mathrm{NN}}\coloneqq C_{\mathrm{traj}}^{\mathrm{NN}}+C_{\mathrm{int}}^{\mathrm{NN}}$.
\end{proposition}

\begin{proof}[Proof of Proposition~\ref{prop:efficiency-condition}]
  Subtracting the two directional-pass costs gives
  \begin{align}
    {C}_{\mathrm{direct}}-{C}_{\mathrm{prop}}
     & =
    M_{\mathrm{SDE}}C_{\mathrm{int}}^{\mathrm{ana}}
    -M_{\mathrm{int}}(C_{\mathrm{int}}^{\mathrm{ana}}+C_{\mathrm{traj}}^{\mathrm{NN}}+C_{\mathrm{int}}^{\mathrm{NN}}) \notag \\
     & =
    (M_{\mathrm{SDE}}-M_{\mathrm{int}})C_{\mathrm{int}}^{\mathrm{ana}}
    -M_{\mathrm{int}}C^{\mathrm{NN}} \notag                                                                                  \\
     & =
    M_{\mathrm{SDE}}
    \left[
      \left(1-\frac{M_{\mathrm{int}}}{M_{\mathrm{SDE}}}\right)C_{\mathrm{int}}^{\mathrm{ana}}
      -\frac{M_{\mathrm{int}}}{M_{\mathrm{SDE}}}C^{\mathrm{NN}}
      \right] \notag                                 \\
     & =
    M_{\mathrm{SDE}}C_{\mathrm{int}}^{\mathrm{ana}}
    \left[
      1
      -\frac{M_{\mathrm{int}}}{M_{\mathrm{SDE}}}
      \left(
      1+\frac{C^{\mathrm{NN}}}{C_{\mathrm{int}}^{\mathrm{ana}}}
      \right)
      \right].
    \label{eq:cost-gap}
  \end{align}
  Since $M_{\mathrm{SDE}}>0$ and $C_{\mathrm{int}}^{\mathrm{ana}}>0$, the last expression is positive exactly when \eqref{eq:efficiency-algebraic-condition} holds.
  Hence \eqref{eq:efficiency-algebraic-condition} is necessary and sufficient for \eqref{eq:efficiency-condition}.
\end{proof}

\begin{remark}[Regimes for the efficiency condition]\label{rem:efficiency-regimes}
  Large populations and infrequent interaction updates make the efficiency condition less restrictive.
  In \eqref{eq:efficiency-algebraic-condition}, infrequent interaction updates reduce the left-hand side $M_{\mathrm{int}}/M_{\mathrm{SDE}}$.
  To examine the population scaling, suppose that the analytical interaction is evaluated by all-pairs sums and that the surrogate is an MLP.
  Let $N$ be the number of particles, $K$ the number of time steps, $d$ the state dimension, and $d_{\mathrm{out}}$ the output dimension of the interaction terms ($d$ for the drift $f$, $1$ for the cost $F$).
  The neural networks are MLPs with $L_{\mathrm{NN}}$ layers and hidden width $H_{\mathrm{NN}}$.
  Then, the dominant costs are approximately
  \begin{align}
    C_{\mathrm{int}}^{\mathrm{ana}}
     & \approx k_{\mathrm{ana}} K N^2 (d+d_{\mathrm{out}}), \\
    C^{\mathrm{NN}}
     & \approx k_{\mathrm{NN}} K N
    \bigl(dH_{\mathrm{NN}}+(L_{\mathrm{NN}}-1)H_{\mathrm{NN}}^2+H_{\mathrm{NN}}d_{\mathrm{out}}\bigr),
  \end{align}
  with hardware- and implementation-dependent constants $k_{\mathrm{ana}}$ and $k_{\mathrm{NN}}$.
  Thus
  \begin{equation}
    \frac{C^{\mathrm{NN}}}{C_{\mathrm{int}}^{\mathrm{ana}}}
    \approx
    \frac{
      k_{\mathrm{NN}}\bigl(dH_{\mathrm{NN}}+(L_{\mathrm{NN}}-1)H_{\mathrm{NN}}^2+H_{\mathrm{NN}}d_{\mathrm{out}}\bigr)
    }{
      k_{\mathrm{ana}}N(d+d_{\mathrm{out}})
    }.
  \end{equation}
  For a fixed architecture, this ratio decreases as $N$ grows, increasing the right-hand side of \eqref{eq:efficiency-algebraic-condition}.
\end{remark}

\section{Stability Analysis}\label{sec:theory}

This section analyzes stability of the FBSDE formulation of the MFSB.
Stability here means that the error between the true system driven by $(f,F)$
and the learned system driven by $(\tilde f,\tilde F)$ can be controlled by the
surrogate approximation errors.
The analysis is carried out at the mean-field level; the finite-particle error
from replacing the marginal laws by empirical measures in
Algorithm~\ref{alg:fourstep_alt_resample} is not included.
We state the assumptions and the total energy estimate in this section, while
the auxiliary estimates and proofs are collected in Appendix~\ref{app:proofs}.

\subsection{Setup}
\label{subsec:stability_setup}

The error analysis compares the ``true'' system driven by $f$ and $F$ with the ``learned'' system driven by the neural surrogates.
Both systems live on a common filtered probability space $(\Omega,\mathcal F,(\mathcal F_t)_{t\in[0,T]},\mathbb P)$ carrying a $d$-dimensional Brownian motion $W$, and all stochastic processes below are adapted to this filtration.

\begin{assumption}[Consistent MFSB solution]
  \label{ass:global-consistency}
  There exists a collection of stochastic processes
  \[
    (X,Y,Z)\quad\text{and}\quad(\widehat X,\widehat Y,\widehat Z),
  \]
  defined on a common filtered probability space, such that:
  \begin{enumerate}[label=(\roman*)]
    \item $(X,Y,Z)$ is the unique strong solution of the forward-time FBSDE
          \eqref{eq:fbsde-forward-full-x}--\eqref{eq:fbsde-forward-full-y}
          driven by the true interaction terms $(f,F)$ with initial condition
          $X_0\sim\rho_0$.

    \item $(\widehat X,\widehat Y,\widehat Z)$ is the unique strong solution of the backward-time FBSDE
          \eqref{eq:fbsde-backward-full-x}--\eqref{eq:fbsde-backward-full-yhat}
          driven by the same interaction terms $(f,F)$ with terminal condition
          $\widehat X_T\sim\rho_T$.

    \item The forward and backward processes are \emph{globally consistent} in the sense that
          they induce the same family of time-marginal distributions $\{\rho_t\}_{t\in[0,T]}$:
          \[
            X_t\sim\rho_t,
            \qquad
            \widehat X_t\sim\rho_t,
            \qquad \forall\,t\in[0,T],
          \]
          and the associated potentials satisfy the Schr\"odinger factorization
          \[
            \rho_t(x)=\exp\bigl(-Y(x,t)-\widehat Y(x,t)\bigr).
          \]

    \item The logarithmic potentials are taken in the endpoint representative
          \[
            Y(\cdot,0)=-\log\rho_0 \quad \rho_0\text{-a.e.},
            \qquad
            \widehat Y(\cdot,T)=-\log\rho_T \quad \rho_T\text{-a.e.}.
          \]
          Together with the factorization in (iii), this implies
          \[
            \widehat Y(\cdot,0)=0 \quad \rho_0\text{-a.e.},
            \qquad
            Y(\cdot,T)=0 \quad \rho_T\text{-a.e.}.
          \]
  \end{enumerate}
\end{assumption}

\begin{assumption}[Consistent learned solution]
  \label{ass:global-consistency-learned}
  There exists a collection of stochastic processes
  \[
    (\tilde X,\tilde Y,\tilde Z)\quad\text{and}\quad
    (\tilde{\widehat X},\tilde{\widehat Y},\tilde{\widehat Z}),
  \]
  defined on the same filtered probability space, such that:
  \begin{enumerate}[label=(\roman*)]
    \item $(\tilde X,\tilde Y,\tilde Z)$ is the unique strong solution of the forward-time FBSDE
          \eqref{eq:fbsde-forward-full-x}--\eqref{eq:fbsde-forward-full-y}
          with interaction terms $(\tilde f,\tilde F)$,
          driven by the same Brownian motion as $(X,Y,Z)$,
          and with the coupled initial condition $\tilde X_0 = X_0$ a.s.\ (so $\tilde X_0\sim\rho_0$).

    \item $(\tilde{\widehat X},\tilde{\widehat Y},\tilde{\widehat Z})$ is the unique strong solution
          of the backward-time FBSDE
          \eqref{eq:fbsde-backward-full-x}--\eqref{eq:fbsde-backward-full-yhat}
          with interaction terms $(\tilde f,\tilde F)$,
          driven by the same Brownian motion as $(\widehat X,\widehat Y,\widehat Z)$,
          and with the coupled terminal condition $\tilde{\widehat X}_T = \widehat X_T$ a.s.\ (so $\tilde{\widehat X}_T\sim\rho_T$).

    \item Defining the learned time-marginal density $\tilde\rho_t$ by
          \[
            \tilde X_t\sim\tilde\rho_t,
            \qquad
            \tilde{\widehat X}_t\sim\tilde\rho_t,
            \qquad \forall\,t\in[0,T],
          \]
          the learned forward and backward processes are assumed to be consistent in the sense that
          the learned potentials satisfy
          \[
            \tilde\rho_t(x)
            =
            \exp\bigl(-\tilde Y(x,t)-\tilde{\widehat Y}(x,t)\bigr),
            \qquad \forall\,t\in[0,T].
          \]

    \item The learned logarithmic potentials are taken in the same endpoint representative:
          \[
            \tilde Y(\cdot,0)=-\log\rho_0 \quad \rho_0\text{-a.e.},
            \qquad
            \tilde{\widehat Y}(\cdot,T)=-\log\rho_T \quad \rho_T\text{-a.e.}.
          \]
          Together with (i)--(iii), this implies
          \[
            \tilde{\widehat Y}(\cdot,0)=0 \quad \rho_0\text{-a.e.},
            \qquad
            \tilde Y(\cdot,T)=0 \quad \rho_T\text{-a.e.}.
          \]
  \end{enumerate}
\end{assumption}

\begin{remark}[Consistency assumptions]
  Assumption~\ref{ass:global-consistency} states that a globally consistent MFSB solution exists;
  existence in variational settings is established by Backhoff et al.~\cite{Backhoff2020Mean}.
  Assumption~\ref{ass:global-consistency-learned} describes the learned system after the interaction terms have been replaced by their surrogates and the resulting forward and backward components are mutually consistent.
  Assumptions~\ref{ass:global-consistency}(iv) and \ref{ass:global-consistency-learned}(iv) fix the endpoint representative of the FBSDE potentials, and thereby make the comparison of the potential variables well defined.
  Indeed, the endpoint factorization \eqref{eq:bc-psi} determines only the sum \(Y+\widehat Y\), so one degree of freedom remains in the endpoint values of the logarithmic potentials.
  This representative is the one singled out by the TD boundary losses
  \eqref{eq:td-loss-bwd}--\eqref{eq:td-loss-fwd}; when these boundary residuals
  are small and the learned factorization is consistent, the learned potentials
  are expected to satisfy Assumption~\ref{ass:global-consistency-learned}(iv).
\end{remark}

We compare the true and learned solutions through the forward-time differences
\[
  \Delta X_t = X_t-\tilde X_t,\quad
  \Delta Y_t = Y_t-\tilde Y_t,\quad
  \Delta Z_t = Z_t-\tilde Z_t,
\]
and the forward total energy error
\begin{equation}\label{eq:total-energy}
  \mathcal E^{\mathrm{fwd}}
  \coloneqq
  \sup_{t\in[0,T]}\E\|\Delta X_t\|^2
  +\sup_{t\in[0,T]}\E|\Delta Y_t|^2
  +\int_0^T\E\|\Delta Z_t\|^2\,dt.
\end{equation}
Analogously, we define the backward-time differences
\[
  \Delta\widehat X_s=\widehat X_s-\tilde{\widehat X}_s,\quad
  \Delta\widehat Y_s=\widehat Y_s-\tilde{\widehat Y}_s,\quad
  \Delta\widehat Z_s=\widehat Z_s-\tilde{\widehat Z}_s,
\]
and the backward total energy error $\mathcal E^{\mathrm{bwd}}$ as
\begin{align}\label{eq:total-energy-backward}
  \mathcal E^{\mathrm{bwd}} \coloneqq \sup_{s\in[0,T]}\E\|\Delta \widehat X_s\|^2 + \sup_{s\in[0,T]}\E|\Delta \widehat Y_s|^2 + \int_0^T \E\|\Delta \widehat Z_s\|^2\,\di s.
\end{align}

We measure the surrogate quality by the maximum deviations
\begin{equation}\label{eq:coeferr_JP}
  \varepsilon_f \coloneqq \sup_{x,t} \norm{f(x,t,\rho_t)-\tilde f(x,t)}, \qquad
  \varepsilon_F \coloneqq \sup_{x,t} |F(x,t,\rho_t)-\tilde F(x,t)|,
\end{equation}
where $\rho_t$ is the marginal of the true solution.

We bound $\mathcal E^{\mathrm{fwd}}$ and $\mathcal E^{\mathrm{bwd}}$ explicitly in terms of the surrogate errors.
We impose the following regularity assumptions on the interaction functions and the FBSDE solutions.

\begin{assumption}[Lipschitz interaction terms]\label{ass:regularity-0}
  The functions $f,\tilde f$ and $F,\tilde F$
  are uniformly Lipschitz in the spatial variable.
  Namely, there exist constants $L_f>0$ and $L_F>0$ such that, for all $t\in[0,T]$,
  all probability densities $\rho_t\in\mathcal{P}(\R^d)$, and all $x,x'\in\R^d$,
  \begin{align}
    \|f(x,t,\rho_t)-f(x',t,\rho_t)\| & \le L_f\|x-x'\|, \\
    |F(x,t,\rho_t)-F(x',t,\rho_t)|   & \le L_F\|x-x'\|,
  \end{align}
  and the same bounds hold for $\tilde f$ and $\tilde F$ with the same constants.
\end{assumption}

\begin{assumption}[Uniform gradient bounds]
  \label{ass:grad-bound-global}
  There exists a positive constant $C_Z$ such that the following uniform bounds hold:
  \begin{align}
     & \sup_{x,t}\|\nabla_x Y(x,t)\| \le C_Z,
    \qquad
    \sup_{x,t}\|\nabla_x \tilde Y(x,t)\| \le C_Z, \label{eq:grad-bound-fwd} \\
     & \sup_{x,s}\|\nabla_x \widehat Y(x,s)\| \le C_Z,
    \qquad
    \sup_{x,s}\|\nabla_x \tilde{\widehat Y}(x,s)\| \le C_Z. \label{eq:grad-bound-bwd}
  \end{align}
\end{assumption}

\begin{remark}[Regularity assumptions]
  For kernel interactions of the form \eqref{eq:nonlocal-form},
  Assumption~\ref{ass:regularity-0} follows when the kernels are uniformly
  Lipschitz in the evaluation variable.
  For the neural surrogates, the same condition corresponds to working in a
  network class with uniformly controlled Lipschitz constants, for instance by
  bounding the weights or using spectral normalization.
  Assumption~\ref{ass:grad-bound-global} is an explicit stability assumption on
  the logarithmic Schr\"odinger potentials, ruling out blow-up of their spatial
  gradients in the estimates below.
\end{remark}

\subsection{Results}

Fix parameters $\alpha,\beta>0$, $\eta\in(0,2)$, and $\gamma>0$, and define
\begin{align}
   & C_{XX}\coloneqq \frac{(L_f)^2}{\alpha} + 3\alpha,\quad
  C_{XZ}\coloneqq \frac{\sigma^2}{\alpha},\quad
  C_{X,\varepsilon_f}\coloneqq \frac{1}{\alpha},                                      \notag  \\
   & C_{YY}\coloneqq \frac{9\beta}{4},\quad
  C_{YZ}\coloneqq 1 + \frac{4\sigma^2 C_Z^2}{\beta},\quad
  C_{YX}\coloneqq \frac{(L_F)^2}{\beta},                                               \notag \\
   & C_{Y,\varepsilon_F}\coloneqq \frac{1}{\beta},\quad
  C_{ZX}\coloneqq \frac{2(L_F)^2}{\gamma},\quad
  C_{ZY}\coloneqq \frac{2\sigma^2 C_Z^2}{\eta} + \gamma,\quad
  C_{Z,\varepsilon_F}\coloneqq \frac{2}{\gamma},\notag                                        \\
   & \kappa\coloneqq \frac{1}{1-\eta/2},\quad
  E_X\coloneqq e^{C_{XX}T},\quad
  E_Y\coloneqq e^{C_{YY}T},\notag                                                             \\
   & a_X\coloneqq E_X\,C_{XZ}\,\kappa\,C_{ZX}\,T,\quad
  a_Y\coloneqq E_X\,C_{XZ}\,\kappa\,C_{ZY}\,T,                                         \notag \\
   & b_X\coloneqq E_Y\Big(C_{YX}\,T + C_{YZ}\,\kappa\,C_{ZX}\,T\Big),\quad
  b_Y\coloneqq E_Y\, C_{YZ}\,\kappa\,C_{ZY}\,T.
  \label{eq:coeff_a_b_def}
\end{align}
Assume
\begin{align}
  D_{\mathrm{det}} \coloneqq (1-a_X)(1-b_Y) - a_Y b_X >0,
  \label{eq:Ddet_def}
\end{align}
and
\begin{align}\label{eq:diag_pos_def}
  1-a_X > 0,
  \qquad
  1-b_Y > 0.
\end{align}
Under these conditions, define
\begin{align}
  \begin{aligned}
    \mathcal C_E
     & \coloneqq
    \frac{1}{D_{\mathrm{det}}}\Bigg[
                                \Big(1+\kappa C_{ZX}T\Big)(1-b_Y)
                                + \Big(1+\kappa C_{ZY}T\Big)b_X
                                \Bigg]\,E_X\Big(C_{X,\varepsilon_f} + C_{XZ}\kappa (1 + C_{Z,\varepsilon_F})\Big)          \\
     & + \frac{1}{D_{\mathrm{det}}}\Bigg[
                                     \Big(1+\kappa C_{ZX}T\Big)a_Y
                                     + \Big(1+\kappa C_{ZY}T\Big)(1-a_X)
                                     \Bigg]\,E_Y\Big(1 + C_{Y,\varepsilon_F} + C_{YZ}\kappa (1 + C_{Z,\varepsilon_F})\Big) \\
     & + \kappa (1 + C_{Z,\varepsilon_F}).
  \end{aligned}\label{eq:CE_def}
\end{align}
The main estimate controls the combined forward--backward error by the surrogate error.

\begin{theorem}[Total energy bound]
  \label{thm:combined_total_energy_bound}
  Assume that Assumptions~\ref{ass:global-consistency},
  \ref{ass:global-consistency-learned},
  \ref{ass:regularity-0}, and \ref{ass:grad-bound-global} hold.
  Assume also \eqref{eq:Ddet_def}--\eqref{eq:diag_pos_def}.
  Then, the combined forward--backward energy satisfies
  \begin{align}\label{eq:combined_energy_bound}
    \mathcal E^{\mathrm{fwd}}+\mathcal E^{\mathrm{bwd}}
    \le
    2T\,\mathcal C_E
    \bigl(\varepsilon_f^2+\varepsilon_F^2\bigr).
  \end{align}
\end{theorem}

The proof of Theorem~\ref{thm:combined_total_energy_bound} has three steps.
First, the forward-time FBSDE gives differential inequalities for
$\E\|\Delta X_t\|^2$ and $\E|\Delta Y_t|^2$, together with an energy estimate for
$\int_0^T\E\|\Delta Z_t\|^2\,\di t$.
Solving these estimates yields the forward bound in
Lemma~\ref{thm:explicit_total_error_bound}.
Second, the same argument is applied to the backward-time FBSDE after reversing
time, giving Lemma~\ref{thm:explicit_total_error_bound_backward_compact}.
Finally, the two estimates are added, leaving the four endpoint errors collected
in the intermediate bounds; these endpoint errors vanish under Assumptions~\ref{ass:global-consistency}(iv) and
\ref{ass:global-consistency-learned}(iv).
These estimates and their proofs are collected in Appendix~\ref{app:proofs}.

\begin{remark}[Small-gain condition]
  Conditions \eqref{eq:Ddet_def}--\eqref{eq:diag_pos_def} are the small-gain
  requirements needed to close the two coupled estimates for
  $\sup_t\E\|\Delta X_t\|^2$ and $\sup_t\E|\Delta Y_t|^2$.
  After the $\Delta Z$ estimate is substituted, these two quantities satisfy a
  two-by-two system of inequalities, displayed in
  \eqref{eq:ineq_Xstar_closed}--\eqref{eq:ineq_Ystar_closed}.
  The conditions state that the diagonal margins in this system are positive
  and that the system can be solved with positive control of the error.
  For fixed Young parameters $\alpha,\beta,\eta,\gamma$, each feedback
  coefficient $a_X,a_Y,b_X,b_Y$ is proportional to $T$, and hence the
  conditions are satisfied on sufficiently short horizons.
  For a fixed horizon, with $C_Z$ fixed, the same conditions are satisfied when
  $L_f$, $L_F$, and $\sigma$ are sufficiently small, with $\gamma$ chosen so
  that both $\gamma$ and $L_F^2/\gamma$ are small.
\end{remark}

\section{Numerical Experiments}\label{sec:experiments}

We evaluate the proposed method on two navigation tasks in $d=2$ and one opinion-dynamics task in $d=10$, all with nonlocal interactions:
\begin{enumerate}
  \item Gaussian Mixture Model (GMM): transport from an initial distribution to an 8-component Gaussian mixture while avoiding obstacles, with an isotropic pairwise drift.
  \item V-neck bottleneck: transport between Gaussian distributions through a narrow V-shaped channel, with a directional congestion cost.
  \item Opinion dynamics: transport between Gaussian opinion distributions under a smooth bounded-confidence drift interaction.
\end{enumerate}
We compare three methods for evaluating $f$ and $F$ along particle trajectories:
(i) baseline FBSDE training with direct analytical evaluation at every step;
(ii) the proposed Algorithm~\ref{alg:fourstep_alt_resample}, which substitutes the learned surrogates $\tilde f$ and $\tilde F$;
(iii) a random-batch (RB) approximation, which retains the analytical formula but evaluates the empirical sum over a random subset of agents at each step.
For each interaction kernel $k_\cdot\in\{k_f,k_F\}$, RB replaces the full empirical average $\frac{1}{N}\sum_{j=1}^N k_\cdot(x,X_j)$
by the mini-batch average $\frac{1}{M_{\mathrm{RB}}}\sum_{j\in\mathcal{B}} k_\cdot(x,X_j)$,
with $\mathcal{B}\subset\{1,\dots,N\}$ a uniformly sampled subset of size $M_{\mathrm{RB}}\ll N$.
Each experiment compares the three methods via particle trajectories, training loss versus wall-clock time, and sensitivity to the interaction strength.
For the last comparison, we summarize each trajectory ensemble by the kernel-affinity diagnostic
\begin{equation}\label{eq:kernel_affinity_diagnostic}
  A_k(t)
  =
  \frac{1}{N(N-1)}
  \sum_{i\ne j}
  k\!\left(X_t^{(i)},X_t^{(j)}\right).
\end{equation}
Here $\{X_t^{(i)}\}_{i=1}^N$ are the particle positions and $k$ is the nonnegative kernel associated with the interaction in each experiment.
A large value of $A_k(t)$ indicates strong average pairwise affinity under this kernel.

All FBSDE networks ($\tilde Y_\theta, \tilde Z_\theta, \tilde{\widehat Y}_\phi, \tilde{\widehat Z}_\phi$) are MLPs with three hidden layers of 128 units and SiLU activations, optimized with AdamW at learning rates $10^{-3}$ for the value networks and $5\times 10^{-4}$ for the gradient networks.
The surrogate networks $\tilde f_\psi$ and $\tilde F_\psi$ use the same architecture with learning rate $5\times 10^{-4}$.
All loss weights are set to $\lambda_{\mathrm{IPF}}=\lambda_{\mathrm{TD}}=\lambda_{\mathrm{FK}}=\lambda_{\mathrm{int}}=1.0$.
Algorithm~\ref{alg:fourstep_alt_resample} uses $M_{\mathrm{SDE}}=250$ drift-update steps per outer iteration in all experiments.
Computational environments are described in Appendix~\ref{app:computational_environment}.

\subsection{GMM Navigation}\label{subsec:gmm}

\paragraph{Setup}
We use $N=1000$ particles, horizon $T=1.0$, and $K=100$ time steps ($\Delta t=0.01$).
The initial distribution is $\rho_0=\mathcal{N}(0, I)$ and the terminal distribution $\rho_T$ is an 8-component Gaussian mixture arranged on a circle.
The nonlocal drift uses the isotropic Gaussian kernel
\begin{equation}
  k(x,y) = \exp\!\left(-\frac{\|x-y\|^2}{2\sigma_{\mathrm{int}}^2}\right),
\end{equation}
giving
\begin{equation}
  f(x,t,\rho_t) = w \int_{\R^d} k(x,y)(y-x)\rho(y,t) \di y,
\end{equation}
with $w=2.0$ and $\sigma_{\mathrm{int}}=2.0$, which attracts each agent toward its neighbors.
The running cost
\begin{equation}
  F(x,t,\rho_t) = 1500 \sum_{m=1}^3 \max(0, 1.5 - \|x - c_m\|)^6,
\end{equation}
with centers $c_m \in \{(6,6),(6,-6),(-6,-6)\}$, penalizes proximity to three fixed obstacles.

\paragraph{Results}
Figure~\ref{fig:gmm_trajectory_comparison} shows particle trajectories for both methods: Figure~\ref{fig:gmm_traj_analytic} uses analytical evaluation of $f$ and $F$, and Figure~\ref{fig:gmm_traj_learned} uses the learned surrogate.
The trajectories in Figure~\ref{fig:gmm_traj_learned} closely match those in Figure~\ref{fig:gmm_traj_analytic}, so the learned surrogate preserves the transport pattern obtained with analytical evaluation.
Figure~\ref{fig:gmm_total_loss_time} plots the total loss $\mathcal{L}_{\mathrm{tot}}$ evaluated with analytical $f$ and $F$ against wall-clock time at $N=1000$, and the learned surrogate reaches the same loss level substantially faster.
Figure~\ref{fig:gmm_int_loss} plots the interaction matching loss $\mathcal{L}_{\mathrm{int}}^N$; it decreases rapidly and settles at a small value, showing that $\tilde f_\psi$ tracks $f$ on the sampled trajectories.
Figure~\ref{fig:gmm_time_comparison} shows how the per-iteration time scales with $N$: the analytical cost grows steeply, the learned surrogate grows only mildly, and the gap exceeds an order of magnitude at the largest $N$ tested.
Figure~\ref{fig:gmm_strength_comparison} compares the diagnostic $A_k(t)$ from analytical and surrogate trajectories for several values of $w$.
The learned surrogate reproduces the qualitative change in $A_k(t)$ as $w$ varies.
Table~\ref{tab:gmm_backend_comparison} compares the learned surrogate ($M_{\mathrm{int}}=200$) and the random-batch approximation ($M_{\mathrm{RB}}=128$).
The learned surrogate attains lower time per iteration and lower best analytical loss, while the random-batch approximation gives a slightly smaller pointwise interaction error.

\begin{figure}[ht]
  \centering
  \begin{subfigure}{0.8\linewidth}
    \centering
    \includegraphics[width=\linewidth]{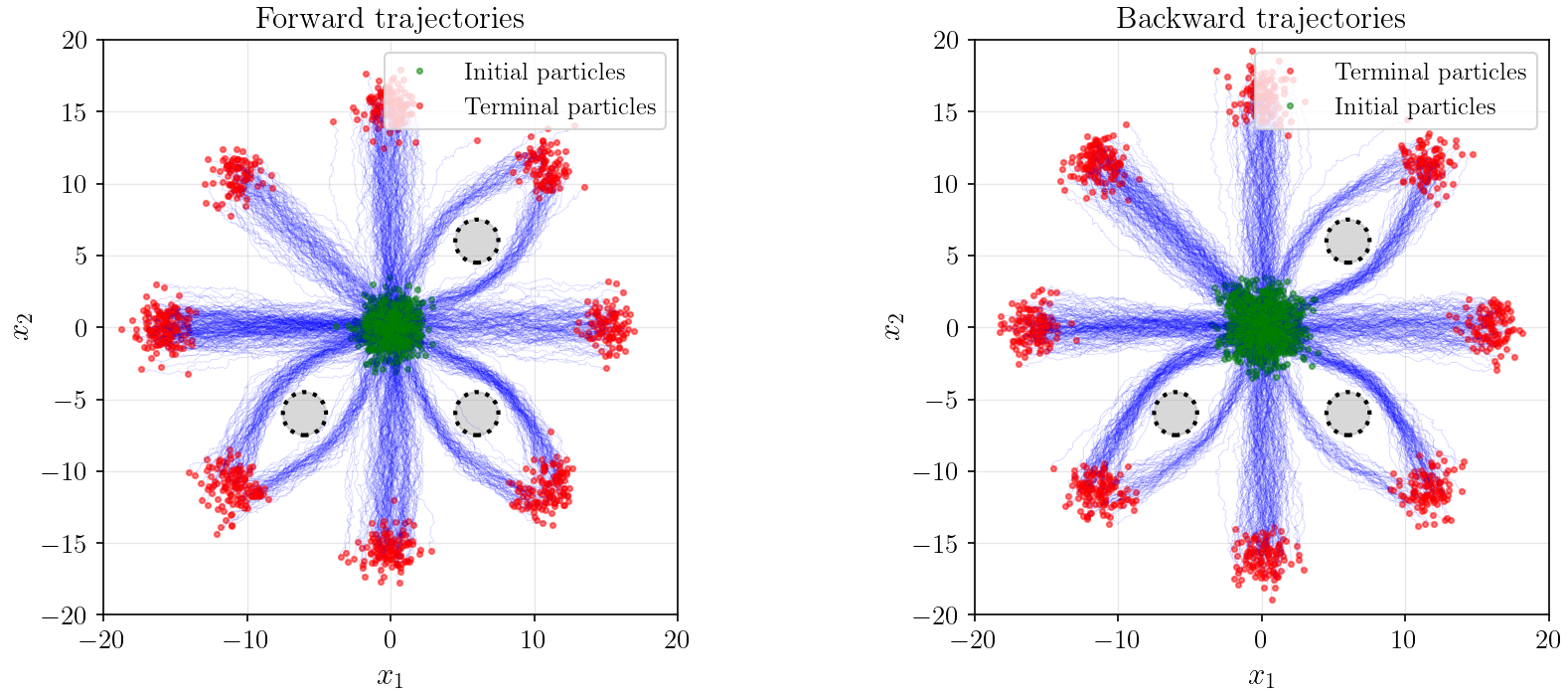}
    \caption{Analytical evaluation}
    \label{fig:gmm_traj_analytic}
  \end{subfigure}
  \hfill
  \begin{subfigure}{0.8\linewidth}
    \centering
    \includegraphics[width=\linewidth]{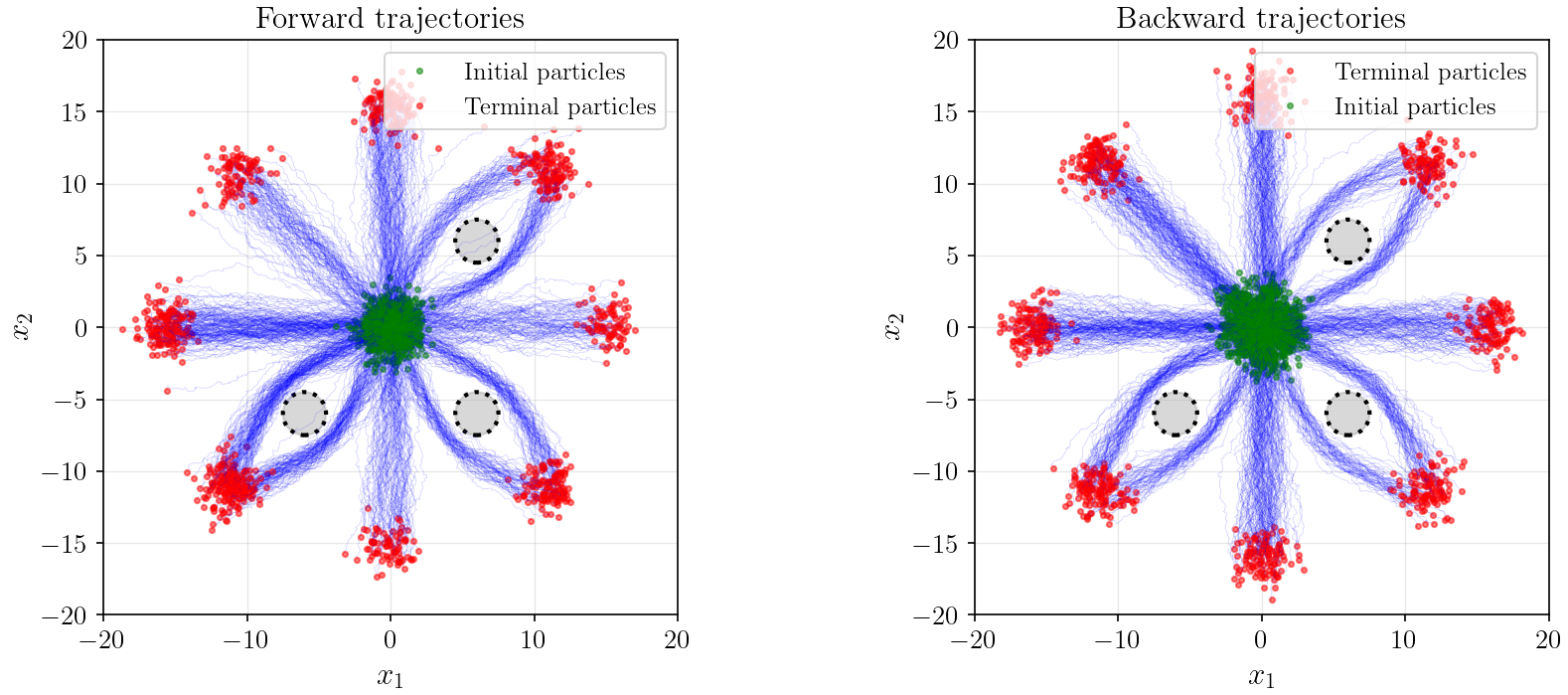}
    \caption{Learned surrogate}
    \label{fig:gmm_traj_learned}
  \end{subfigure}
  \caption{GMM: trajectory comparison (forward left, backward right). The learned surrogate reproduces the trajectories obtained with analytical evaluation.}
  \label{fig:gmm_trajectory_comparison}
\end{figure}

\begin{figure}[ht]
  \centering
  \includegraphics[width=0.6\linewidth]{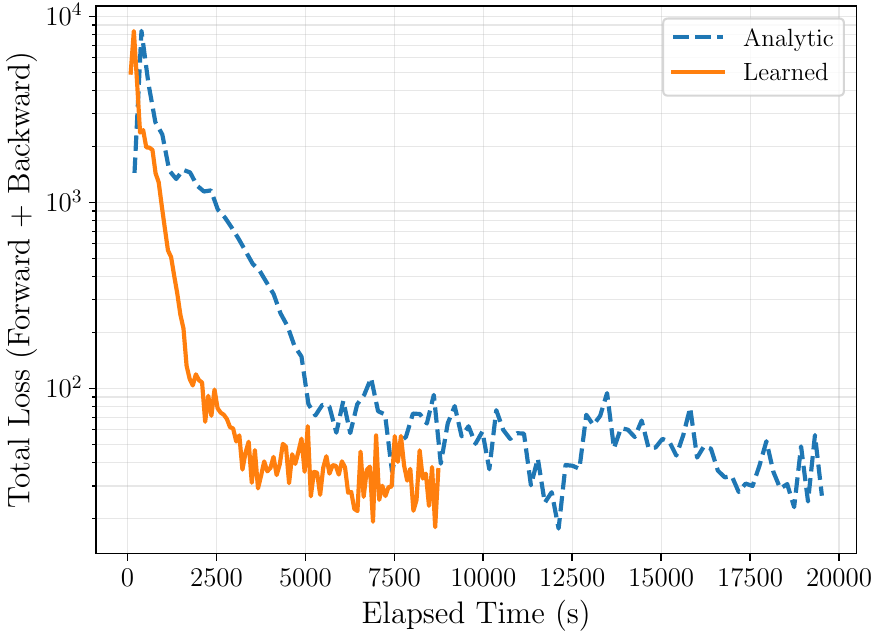}
  \caption{GMM: total loss evaluated with analytical $f$ and $F$ vs wall-clock time.}
  \label{fig:gmm_total_loss_time}
\end{figure}

\begin{figure}[ht]
  \centering
  \includegraphics[width=0.6\linewidth]{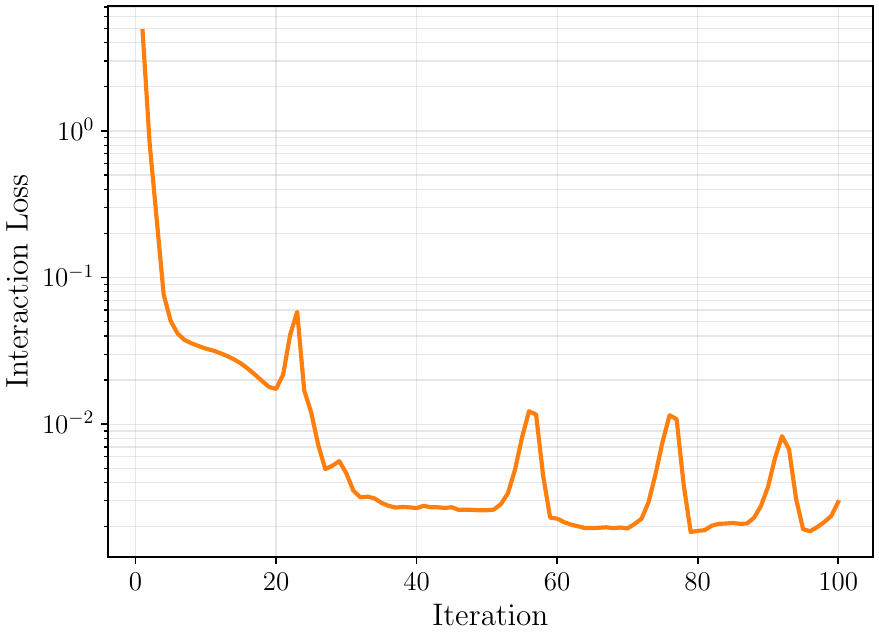}
  \caption{GMM: interaction matching loss $\mathcal{L}_{\mathrm{int}}^N$.}
  \label{fig:gmm_int_loss}
\end{figure}

\begin{figure}[ht]
  \centering
  \includegraphics[width=0.6\linewidth]{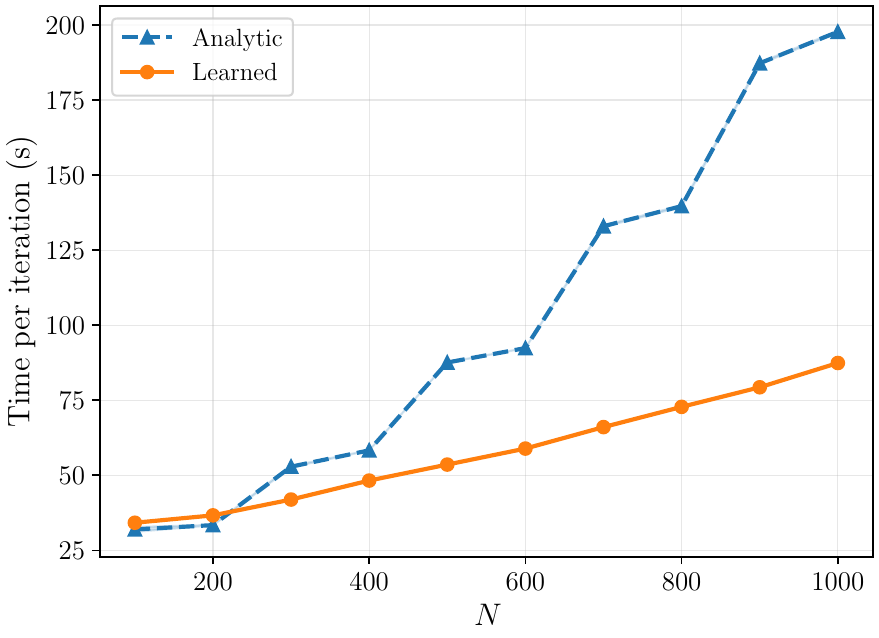}
  \caption{GMM: wall-clock time per outer iteration vs number of agents $N$.}
  \label{fig:gmm_time_comparison}
\end{figure}

\begin{figure}[ht]
  \centering
  \begin{subfigure}{0.49\linewidth}
    \centering
    \includegraphics[width=\linewidth]{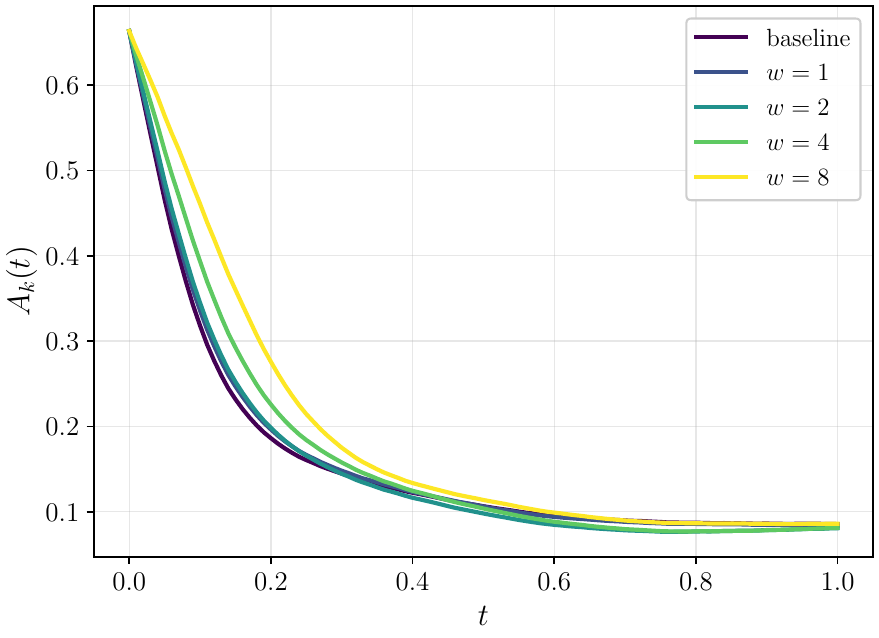}
    \caption{Analytical evaluation.}
  \end{subfigure}
  \hfill
  \begin{subfigure}{0.49\linewidth}
    \centering
    \includegraphics[width=\linewidth]{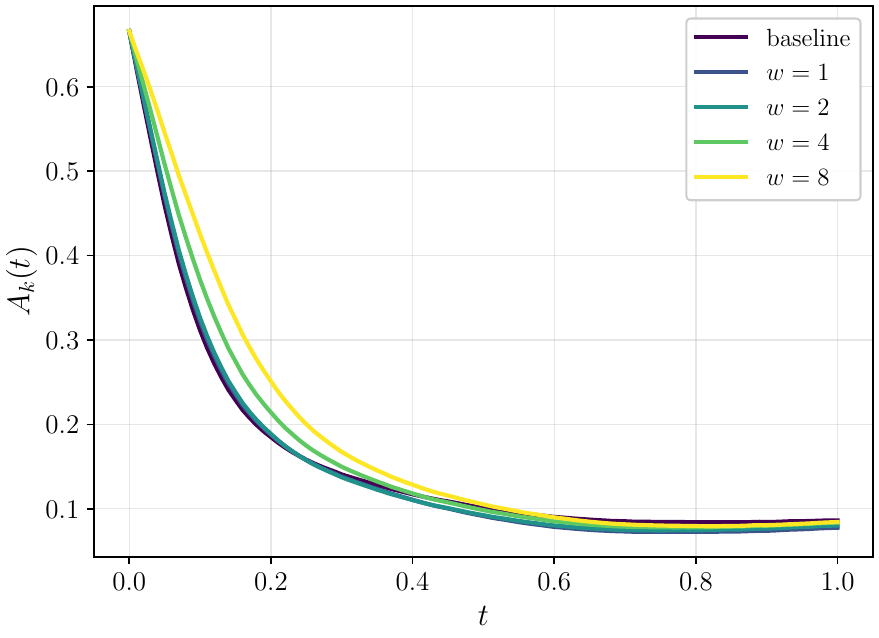}
    \caption{Learned surrogate.}
  \end{subfigure}
  \caption{GMM: proximity diagnostic $A_k(t)$ across values of $w$. Each panel shows the case $w=0$ and several nonzero values of $w$.}
  \label{fig:gmm_strength_comparison}
\end{figure}

\begin{table}[ht]
  \centering
  \begin{tabular}{lccc}
    \hline
    Method                                     & Time/iter. (s)   & Analytical loss & Interaction error \\
    \hline
    Learned surrogate ($M_{\mathrm{int}}=200$) & $93.66 \pm 1.36$ & $3.78 \pm 0.83$ & $0.229 \pm 0.040$ \\
    Random batch ($M_{\mathrm{RB}}=128$)       & $99.10 \pm 0.95$ & $4.48 \pm 1.36$ & $0.208 \pm 0.014$ \\
    \hline
  \end{tabular}
  \caption{GMM backend comparison at $N=1000$, $\Delta t=0.01$. All columns are averaged over five independent runs and reported as mean $\pm$ standard deviation. Time per iteration is the total training time divided by the 100 outer iterations. Analytical loss is the best checkpoint-evaluated value of $\mathcal{L}_{\mathrm{tot}}$, computed with analytical interactions. Interaction error is the mean relative $L^2$ error of the forward and backward drift interaction.}
  \label{tab:gmm_backend_comparison}
\end{table}

\subsection{V-neck Navigation with Directional Congestion}\label{subsec:vneck}
\paragraph{Setup}
We consider transport through a narrow V-shaped bottleneck with a directional congestion cost.
The initial distribution is $\rho_0=\mathcal{N}((-7,0)^\top,0.2I_2)$ and the terminal distribution is $\rho_T=\mathcal{N}((7,0)^\top,0.2I_2)$.
A constant base drift $f_0(x)=(6,0)^\top$ pushes particles toward the target, while the obstacle cost
\begin{equation}
  F_{\mathrm{obs}}(x)
  =
  1500\cdot \mathbf{1}\{5x_1^2-x_2^2<-0.36\}
  \label{eq:vneck_directional_obstacle}
\end{equation}
penalizes particles outside the V-shaped passage.
Pedestrian models often combine a target or desired direction with nonlocal crowd avoidance~\cite{Piccoli2009Pedestrian,Colombo2011Class}; social-force models also include direction-dependent responses to nearby pedestrians~\cite{Helbing1995Social}.
Motivated by these modeling principles, we use the direction to the common target as a simple proxy for the walking direction.
Let $g=(7,0)^\top$ and define
\begin{equation}
  e_g(x)=\frac{g-x}{\norm{g-x}+\varepsilon},
  \qquad
  e_{xy}=\frac{y-x}{\norm{y-x}+\varepsilon}.
\end{equation}
The directional congestion kernel is
\begin{equation}
  k(x,y)
  =
  \frac{1+\ip{e_g(x)}{e_{xy}}}{2}
  \exp\!\left(-\frac{\norm{x-y}}{\sigma_{\mathrm{int}}}\right),
  \label{eq:vneck_directional_kernel}
\end{equation}
and the corresponding nonlocal running cost is
\begin{equation}
  F_{\mathrm{int}}(x,t,\rho_t)
  =
  w\int_{\R^2} k(x,y)\rho(y,t)\,\di y.
  \label{eq:vneck_directional_cost}
\end{equation}
The angular factor assigns greater weight to nearby particles located in front of the query particle, yielding a minimal directional-congestion benchmark rather than a calibrated pedestrian model.
Unlike a translation-invariant convolution kernel, $k(x,y)$ depends on the absolute query position through the goal direction $e_g(x)$ and is generally asymmetric in $(x,y)$.
The total running cost is $F=F_{\mathrm{obs}}+F_{\mathrm{int}}$.
We use $N=1000$, horizon $T=2.0$, time step $\Delta t=0.01$, diffusion coefficient $\sigma=1.0$, $w=1.0$, $\sigma_{\mathrm{int}}=1.0$, and $\varepsilon=10^{-6}$.
The learned surrogate uses three hidden layers of 128 units and $M_{\mathrm{int}}=100$ interaction updates per outer iteration.
Its inputs are divided by $7$ and its scalar cost targets are standardized during interaction matching.
For the interaction-strength sweep, we vary $w$ while keeping the remaining parameters fixed.

\paragraph{Results}
Figure~\ref{fig:vneck_trajectory_comparison} compares trajectories produced by analytical evaluation and by the learned surrogate.
Both methods transport the particles through the narrow passage and produce closely matching flow patterns.
Figure~\ref{fig:vneck_total_loss_time} shows the total loss evaluated with the analytical directional congestion cost against wall-clock time.
The learned surrogate reaches a low analytical loss substantially earlier than direct evaluation.
The interaction matching loss in Figure~\ref{fig:vneck_int_loss} decreases throughout training, confirming that the surrogate fits the directional congestion cost along sampled trajectories.
Figure~\ref{fig:vneck_time_comparison} shows that the runtime gap increases with $N$; at $N=900$, direct evaluation takes approximately $958$ seconds per outer iteration, compared with approximately $171$ seconds for the learned surrogate.
Figure~\ref{fig:vneck_strength_comparison} compares $A_k(t)$ from analytical and surrogate trajectories for several values of $w$.
Increasing the congestion weight reduces directional pairwise affinity at intermediate times, and the learned surrogate reproduces both this dependence and the overall temporal profile.
Table~\ref{tab:vneck_backend_comparison} compares the learned surrogate with the random-batch approximation using the same directional congestion cost.
The learned surrogate slightly reduces the time per iteration, while the analytical losses and interaction errors of the two approximations are comparable.

\begin{figure}[ht]
  \centering
  \begin{subfigure}{0.8\linewidth}
    \centering
    \includegraphics[width=\linewidth]{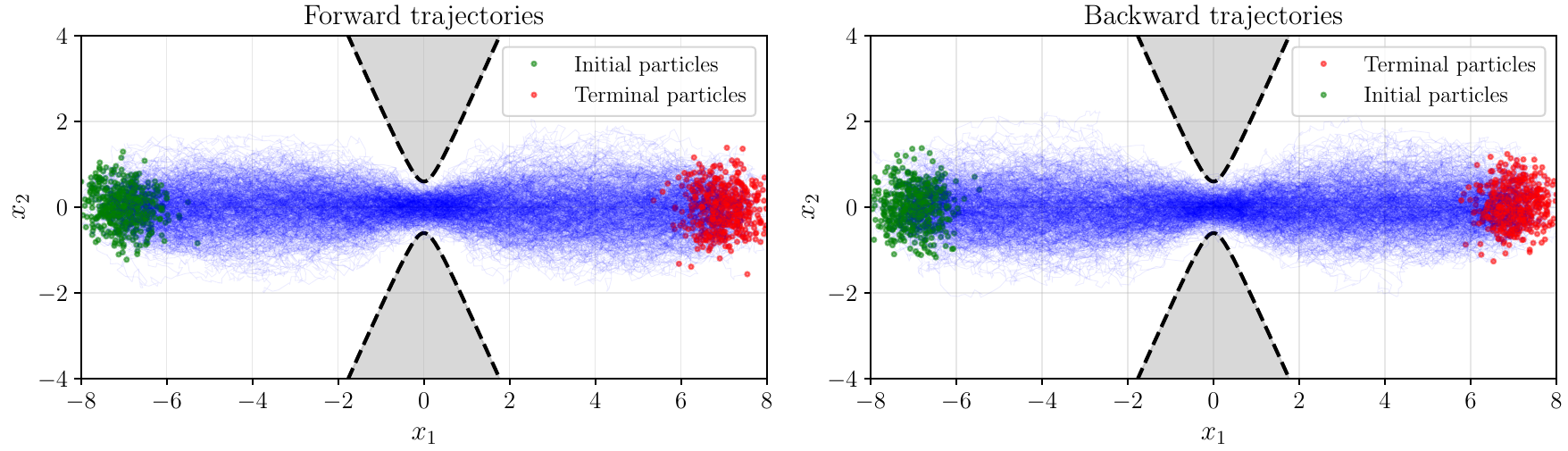}
    \caption{Analytical evaluation}
    \label{fig:vneck_traj_analytic}
  \end{subfigure}
  \hfill
  \begin{subfigure}{0.8\linewidth}
    \centering
    \includegraphics[width=\linewidth]{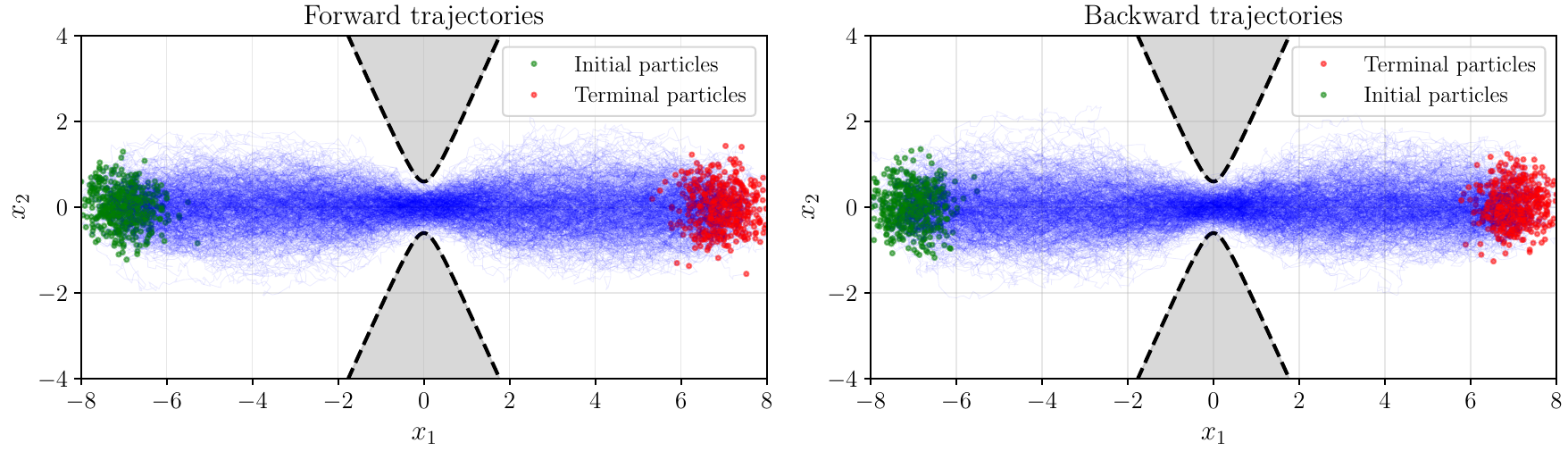}
    \caption{Learned surrogate}
    \label{fig:vneck_traj_learned}
  \end{subfigure}
  \caption{V-neck navigation with directional congestion: trajectory comparison (forward left, backward right).}
  \label{fig:vneck_trajectory_comparison}
\end{figure}

\begin{figure}[ht]
  \centering
  \includegraphics[width=0.6\linewidth]{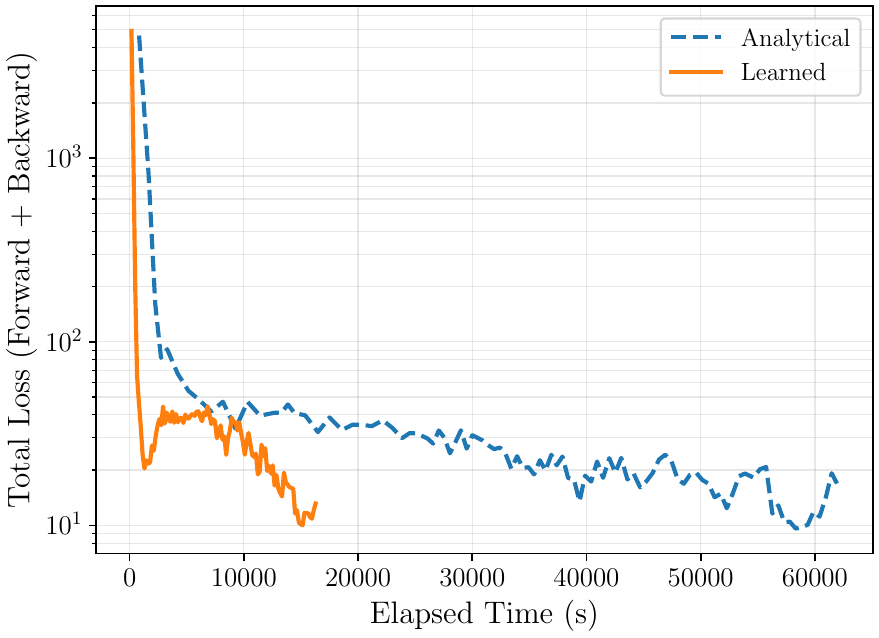}
  \caption{V-neck navigation with directional congestion: total loss evaluated with the analytical interaction vs wall-clock time.}
  \label{fig:vneck_total_loss_time}
\end{figure}

\begin{figure}[ht]
  \centering
  \includegraphics[width=0.6\linewidth]{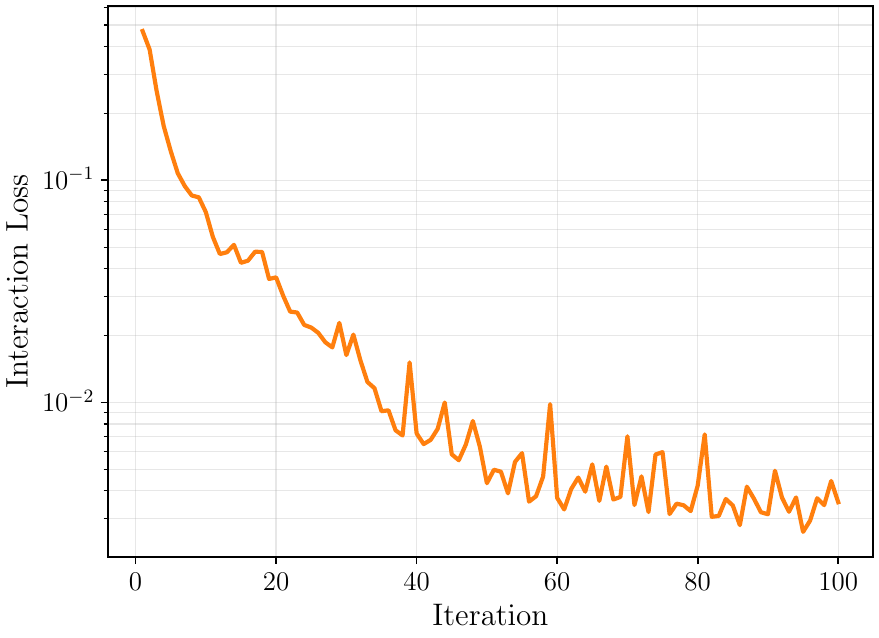}
  \caption{V-neck navigation with directional congestion: interaction matching loss $\mathcal{L}_{\mathrm{int}}^N$.}
  \label{fig:vneck_int_loss}
\end{figure}

\begin{figure}[ht]
  \centering
  \includegraphics[width=0.6\linewidth]{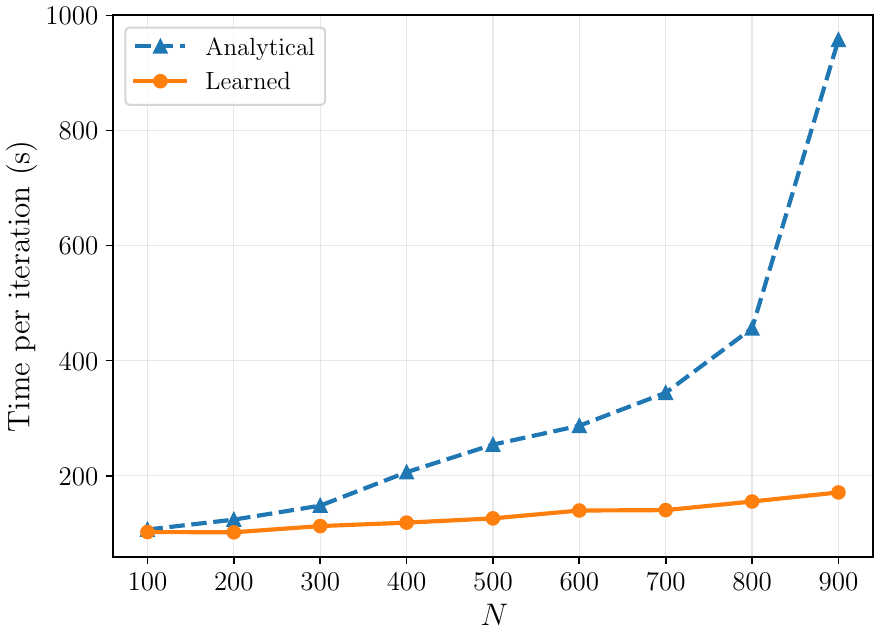}
  \caption{V-neck navigation with directional congestion: wall-clock time per outer iteration vs number of agents $N$.}
  \label{fig:vneck_time_comparison}
\end{figure}

\begin{figure}[ht]
  \centering
  \begin{subfigure}{0.49\linewidth}
    \centering
    \includegraphics[width=\linewidth]{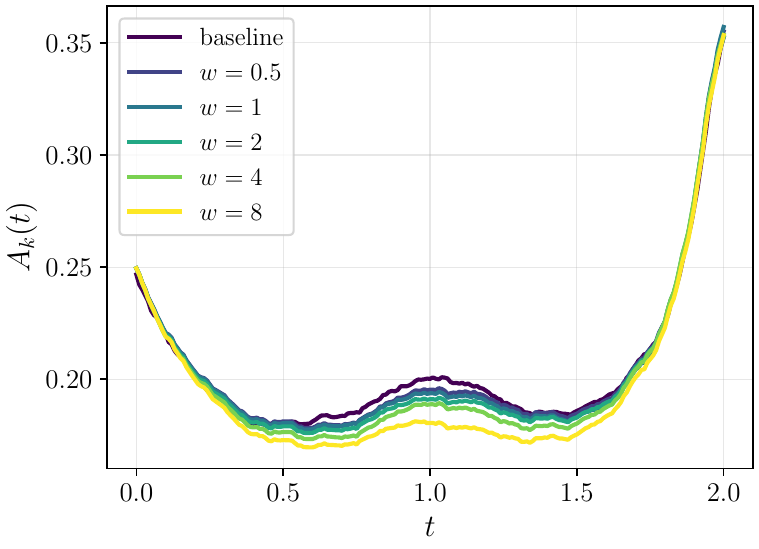}
    \caption{Analytical evaluation.}
  \end{subfigure}
  \hfill
  \begin{subfigure}{0.49\linewidth}
    \centering
    \includegraphics[width=\linewidth]{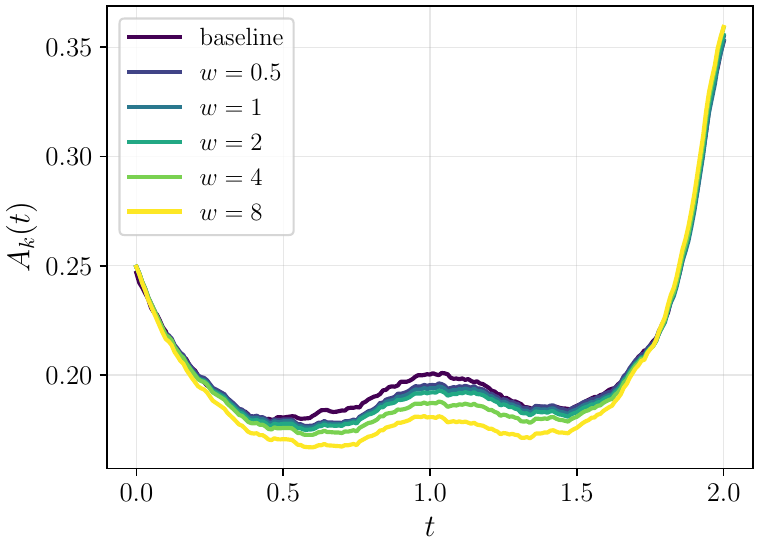}
    \caption{Learned surrogate.}
  \end{subfigure}
  \caption{V-neck navigation with directional congestion: kernel-affinity diagnostic $A_k(t)$ across values of $w$ at $N=1000$. Each panel includes the no-interaction baseline.}
  \label{fig:vneck_strength_comparison}
\end{figure}

\begin{table}[ht]
  \centering
  \begin{tabular}{lccc}
    \hline
    Method                                     & Time/iter. (s)     & Analytical loss  & Interaction error \\
    \hline
    Learned surrogate ($M_{\mathrm{int}}=200$) & $173.62 \pm 16.11$ & $12.45 \pm 5.09$ & $0.062 \pm 0.048$ \\
    Random batch ($M_{\mathrm{RB}}=128$)       & $178.73 \pm 5.16$  & $11.48 \pm 5.26$ & $0.064 \pm 0.001$ \\
    \hline
  \end{tabular}
  \caption{V-neck backend comparison with directional congestion at $N=1000$ and $\Delta t=0.01$. All columns are averaged over eight paired independent runs and reported as mean $\pm$ standard deviation. Time per iteration is the total training time divided by the 100 outer iterations. Analytical loss is the best checkpoint-evaluated value of $\mathcal{L}_{\mathrm{tot}}$, computed with the analytical directional congestion cost. Interaction error is the mean relative $L^2$ error of the forward and backward cost interaction.}
  \label{tab:vneck_backend_comparison}
\end{table}

\subsection{Opinion Dynamics}\label{subsec:opinion_bounded_confidence}
\paragraph{Setup}
Finally, we consider a $d=10$ opinion-dynamics problem in which the interaction enters through the drift.
The initial and terminal opinion distributions are $\rho_0=\mathcal{N}(0,0.25^2 I_d)$ and $\rho_T=\mathcal{N}(0,3.0^2 I_d)$.
The drift follows Hegselmann--Krause-type bounded-confidence dynamics~\cite{Chazelle2017Wellposedness,Bernardo2024Bounded,Garnier2017Consensus}, replacing the discontinuous confidence set~\cite{Ceragioli2011Continuoustime,Ceragioli2012Continuous} with a Gaussian kernel $k(r)=\exp\!\left(-\frac{r^2}{2\sigma_{\mathrm{int}}^2}\right)$.
In the mean-field formulation, the drift interaction is
\begin{equation}
  f(x,t,\rho_t)
  =
  w
  \frac{
    \int_{\R^d} k(\|y-x\|)(y-x)\rho(y,t)\,\di y
  }{
    \int_{\R^d} k(\|y-x\|)\rho(y,t)\,\di y+\varepsilon
  }.
  \label{eq:opinion_smooth_bc}
\end{equation}
Nearby opinions have greater influence on the drift, and the effect decays smoothly with distance.
We use $N=400$, horizon $T=3.0$, time step $\Delta t=0.01$, diffusion coefficient $\sigma=1.0$, $w=1.0$, $\sigma_{\mathrm{int}}=1.5$, and $\varepsilon=10^{-8}$.
Because the drift in \eqref{eq:opinion_smooth_bc} is normalized by the local kernel mass, $A_k(t)$ measures average bounded-confidence connectivity rather than the magnitude of the normalized drift.
Related computational approaches to controlled opinion dynamics include deep kinetic neural networks~\cite{Albi2025Control} and mean-field equilibrium computation for the Hegselmann--Krause model~\cite{Bicego2025Computation}.
Since only $f$ is nonlocal, the algorithm trains a drift surrogate $\tilde f_\psi$ with $M_{\mathrm{int}}=100$ interaction updates.

\paragraph{Results}
Figure~\ref{fig:opinion_trajectory_pca} compares forward-time snapshots after projecting the $10$-dimensional paths onto a common two-dimensional PCA basis.
Without control, the bounded-confidence interaction keeps the population concentrated and does not reach the prescribed high-variance terminal distribution.
The controlled dynamics computed with analytical interactions spreads the population toward the target distribution, and the learned surrogate reproduces this qualitative behavior.
Figure~\ref{fig:opinion_total_loss_time} plots the total loss evaluated with the analytical drift against wall-clock time.
The learned surrogate converges to a lower analytical loss in less wall-clock time than direct analytical evaluation.
Figure~\ref{fig:opinion_int_loss} shows the interaction matching loss for the learned surrogate; the forward and backward drift errors decrease over training, confirming that the surrogate fits the bounded-confidence drift along sampled trajectories.
Figure~\ref{fig:opinion_time_comparison} shows how the per-iteration time scales with the number of particles.
The learned surrogate is substantially faster than direct analytical evaluation across all values of $N$ tested, with the gap already visible at $N=150$.
Figure~\ref{fig:opinion_strength_comparison} shows the kernel-affinity diagnostic $A_k(t)$ at $N=400$ as the interaction weight $w$ varies.
Both analytical evaluation and the learned surrogate show the same qualitative response: larger values of $w$ keep the particles more tightly clustered over a longer portion of the trajectory.
Table~\ref{tab:opinion_backend_comparison} compares the learned surrogate and the random-batch approximation.
The learned surrogate is approximately $2.6$ times faster and achieves substantially lower analytical loss and interaction error than the random-batch approximation.

\begin{figure}[ht]
  \centering
  \includegraphics[width=0.95\linewidth]{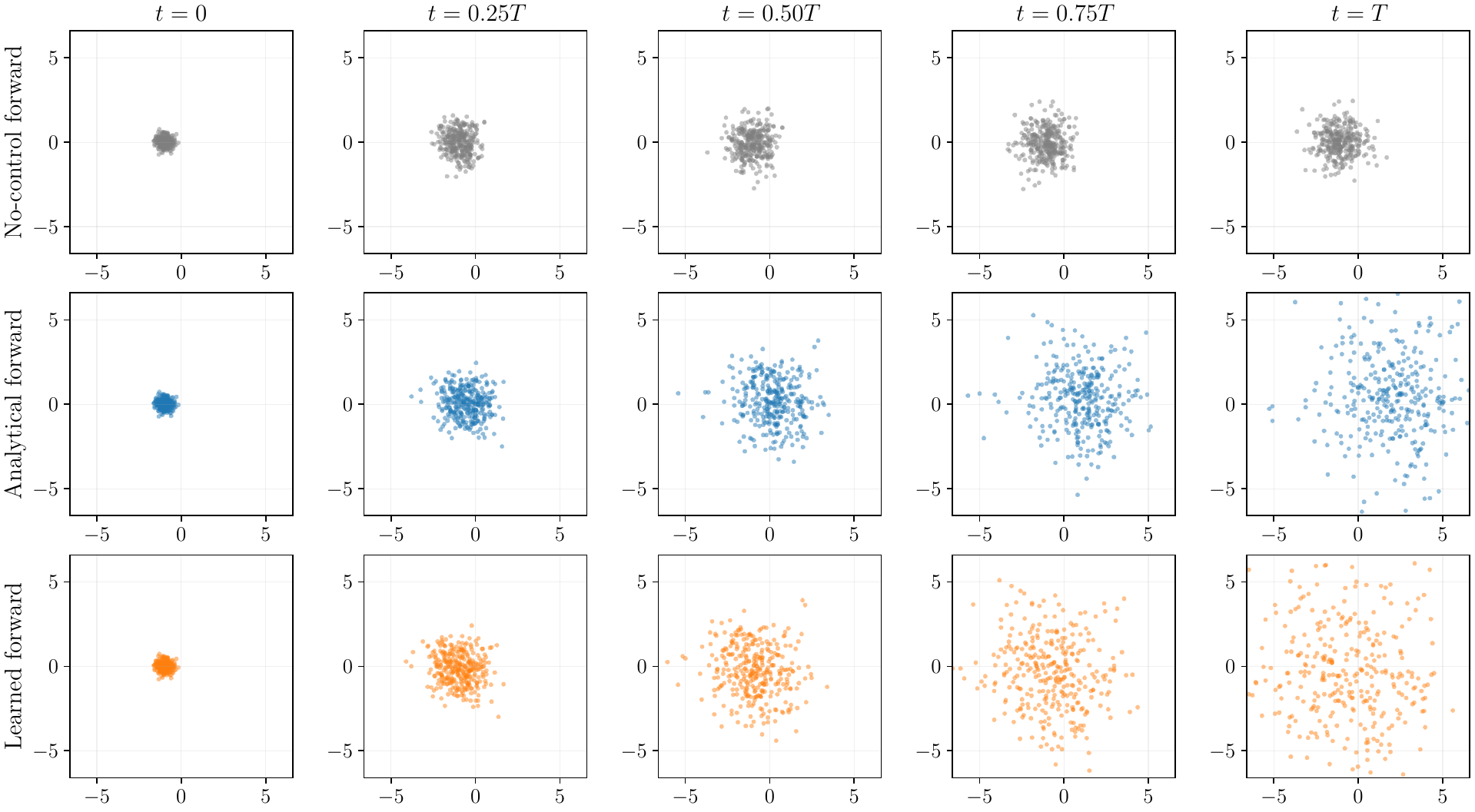}
  \caption{Opinion dynamics: PCA snapshots of uncontrolled and controlled forward trajectories in $d=10$. The uncontrolled process, analytical controlled process, and learned-surrogate controlled process are projected onto the same PCA basis.}
  \label{fig:opinion_trajectory_pca}
\end{figure}

\begin{figure}[ht]
  \centering
  \includegraphics[width=0.6\linewidth]{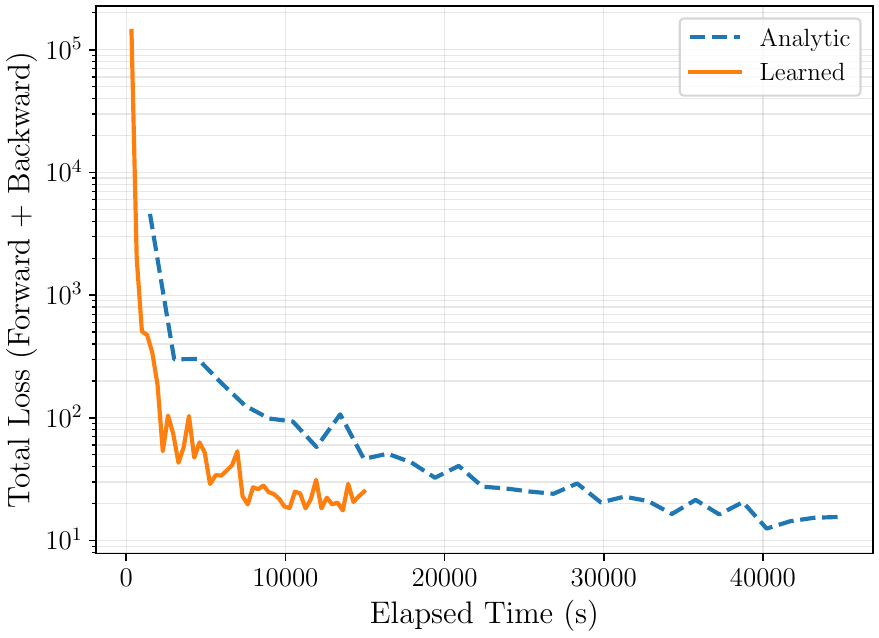}
  \caption{Opinion dynamics: total loss evaluated with analytical drift vs wall-clock time.}
  \label{fig:opinion_total_loss_time}
\end{figure}

\begin{figure}[ht]
  \centering
  \includegraphics[width=0.6\linewidth]{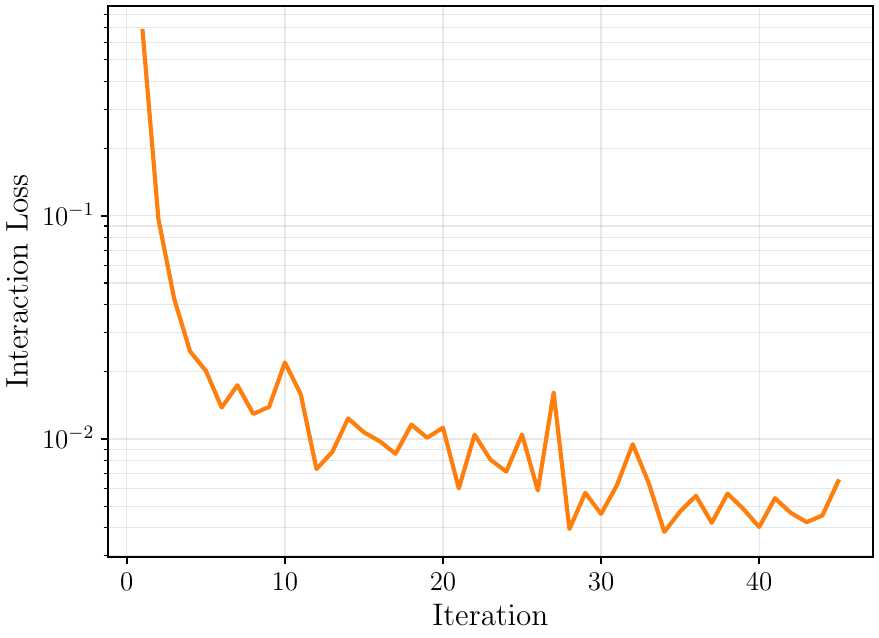}
  \caption{Opinion dynamics: interaction matching loss $\mathcal{L}_{\mathrm{int}}^N$ for the learned drift surrogate.}
  \label{fig:opinion_int_loss}
\end{figure}

\begin{figure}[ht]
  \centering
  \includegraphics[width=0.6\linewidth]{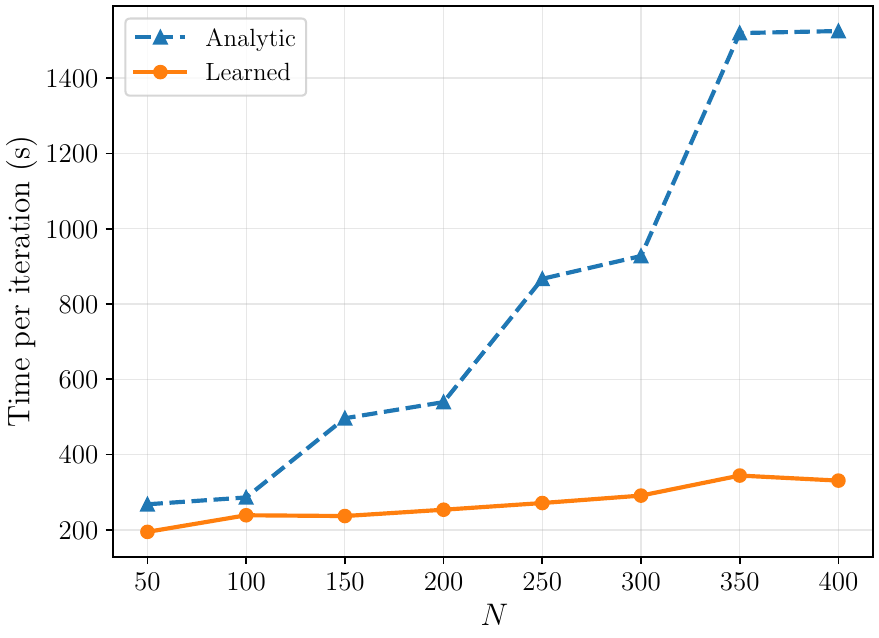}
  \caption{Opinion dynamics: wall-clock time per outer iteration vs number of agents $N$.}
  \label{fig:opinion_time_comparison}
\end{figure}

\begin{figure}[ht]
  \centering
  \begin{subfigure}{0.49\linewidth}
    \centering
    \includegraphics[width=\linewidth]{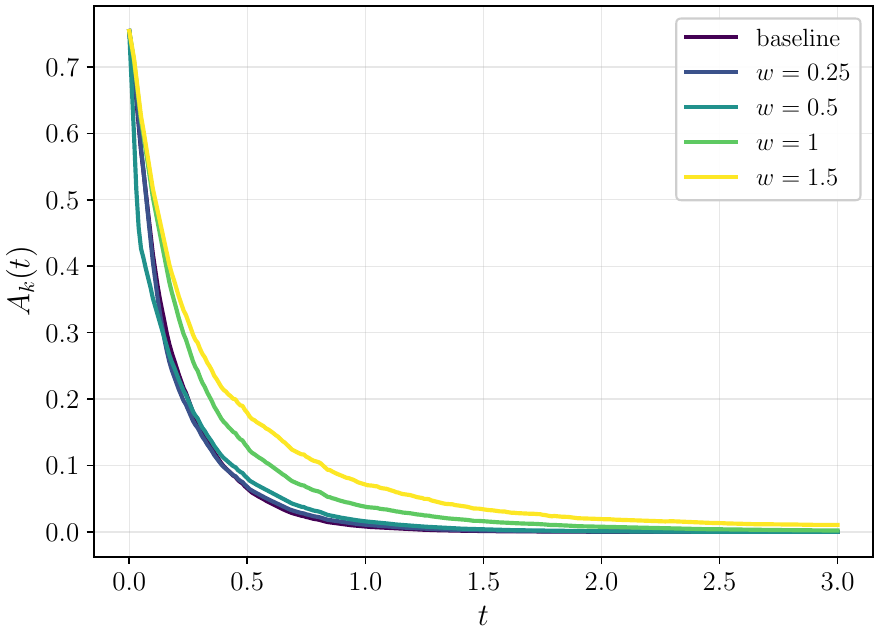}
    \caption{Analytical evaluation.}
  \end{subfigure}
  \hfill
  \begin{subfigure}{0.49\linewidth}
    \centering
    \includegraphics[width=\linewidth]{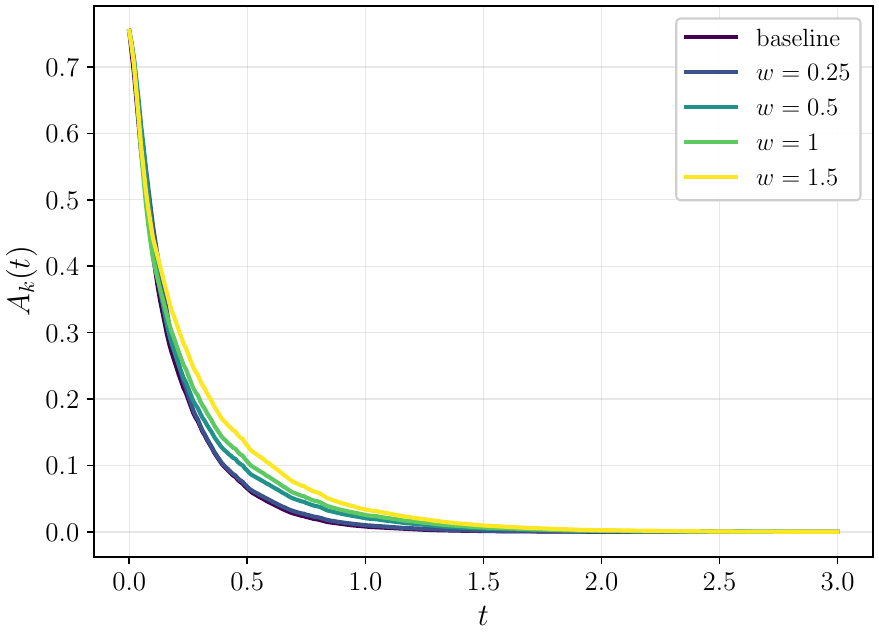}
    \caption{Learned surrogate.}
  \end{subfigure}
  \caption{Opinion dynamics: kernel-affinity diagnostic $A_k(t)$ across values of $w$ at $N=400$. Each panel includes the no-interaction baseline.}
  \label{fig:opinion_strength_comparison}
\end{figure}

\begin{table}[ht]
  \centering
  \begin{tabular}{lccc}
    \hline
    Method                                     & Time/iter. (s)     & Analytical loss     & Interaction error \\
    \hline
    Learned surrogate ($M_{\mathrm{int}}=100$) & $253.10 \pm 57.53$ & $29.80 \pm 20.14$   & $0.228 \pm 0.018$ \\
    Random batch ($M_{\mathrm{RB}}=80$)        & $567.39 \pm 15.83$ & $497.18 \pm 716.37$ & $2.298 \pm 2.177$ \\
    \hline
  \end{tabular}
  \caption{Opinion dynamics backend comparison at $N=400$, $d=10$, and $\Delta t=0.01$. All columns are averaged over six paired independent runs and reported as mean $\pm$ standard deviation. Time per iteration is the total training time divided by the 30 outer iterations. Analytical loss is the best checkpoint-evaluated value of $\mathcal{L}_{\mathrm{tot}}$, computed with the analytical drift. Interaction error is the mean relative $L^2$ error of the forward and backward drift interaction.}
  \label{tab:opinion_backend_comparison}
\end{table}

\section{Conclusion}\label{sec:conclusion}

We addressed the computational bottleneck of solving MFSB problems with nonlocal interactions by replacing the expensive distributional terms with neural surrogates in state and time. The four-stage alternating scheme integrates surrogates of the drift~$f$ and running cost~$F$ into the FBSDE framework, reducing the per-step cost from $O(N^2)$ to $O(N)$ while preserving forward--backward consistency and the prescribed endpoint marginals. The accompanying stability analysis turns this empirical acceleration into a controlled approximation: Gr\"onwall-type bounds show, under a small-gain condition, that the trajectory error is governed by the surrogate error alone.

Across crowd-navigation and high-dimensional opinion-dynamics benchmarks, the surrogates reproduce the trajectories obtained with exact evaluation at lower training cost. The sharpest gains appear where they matter most: when the interaction is a nonlinear functional of the measure, as in the normalized bounded-confidence drift, random-batch subsampling becomes biased and unstable, whereas the learned surrogate stays accurate and several times faster. There is a class of problems in which learned surrogates are not merely a convenience but the more reliable option, and it points to a broader principle for accelerating mean-field solvers: the harder the interaction is to estimate by subsampling, the more there is to gain from learning it.

Several directions remain open. Establishing convergence guarantees for the
alternating training scheme is the most immediate, and quantifying the
finite-$N$ error from replacing the marginal laws by empirical measures would close the gap between the theory and the particle-based algorithm actually run. Beyond this, surrogates that depend explicitly on the distribution, rather than absorbing it through the time argument, may extend the approach to interactions that vary sharply with the
population state.

\appendix
\section{Proofs}\label{app:proofs}

\begin{lemma}[Forward differential inequalities]\label{thm:diff_ineq_XY}
  Assume that Assumptions~\ref{ass:global-consistency}, \ref{ass:global-consistency-learned}, \ref{ass:regularity-0}, and \ref{ass:grad-bound-global} hold.
  Then, for all $t\in[0,T]$, the error processes
  $\Delta X_t = X_t-\tilde X_t$, $\Delta Y_t = Y_t-\tilde Y_t$ satisfy
  \begin{subequations}\label{eq:split-JP-diff-ineq}
    \begin{align}
      \frac{\mathrm d}{\di t}\,\E\norm{\Delta X_t}^2
       & \le
      C_{XX}\,\E\norm{\Delta X_t}^2
      + C_{XZ}\,\E\norm{\Delta Z_t}^2
      + C_{X,\varepsilon_f}\,\varepsilon_f^2,
      \label{eq:split-JP-DeltaX}
      \\
      \frac{\mathrm d}{\di t}\,\E|\Delta Y_t|^2
       & \le
      C_{YY}\,\E|\Delta Y_t|^2
      + C_{YX}\,\E\norm{\Delta X_t}^2
      + C_{YZ}\,\E\norm{\Delta Z_t}^2
      + C_{Y,\varepsilon_F}\,\varepsilon_F^2.
      \label{eq:split-JP-DeltaY}
    \end{align}
  \end{subequations}
\end{lemma}

\begin{proof}[Proof of Lemma~\ref{thm:diff_ineq_XY}]
  First, we focus on the forward difference $\Delta X_t\coloneqq X_t-\tilde X_t$.
  Let
  \begin{align}
    \delta f_t & \coloneqq f(X_t,t,\rho_t) - f(\tilde X_t,t,\rho_t),        \\
    r^f_t      & \coloneqq f(\tilde X_t,t,\rho_t) - \tilde f(\tilde X_t,t),
  \end{align}
  then, we have
  \begin{equation}\label{eq:DeltaX-diff}
    \di\Delta X_t = (\delta f_t + r^f_t)\dt + \sigma\,\Delta Z_t\dt.
  \end{equation}
  Applying It\^{o}'s formula to $\norm{\Delta X_t}^2$ yields
  \begin{equation}\label{eq:dX-squared}
    \di\norm{\Delta X_t}^2 = 2\ip{\Delta X_t}{\delta f_t}\dt + 2\ip{\Delta X_t}{r^f_t}\dt + 2\ip{\Delta X_t}{\sigma\Delta Z_t}\dt.
  \end{equation}
  By Assumption~\ref{ass:regularity-0},
  \begin{align}
    \norm{\delta f_t} \le L_f\norm{\Delta X_t}.
  \end{align}
  Using Young's inequality $2ab\le \alpha a^2 + b^2/\alpha$ (for $\alpha>0$), we obtain
  \begin{align}
    2\ip{\Delta X_t}{\delta f_t}
    \le \left(\frac{(L_f)^2}{\alpha}+\alpha\right)\norm{\Delta X_t}^2.
  \end{align}
  Also, by definition, $r^f_t$ satisfies $\norm{r^f_t}\le \varepsilon_f$, so
  \begin{align}
    2\ip{\Delta X_t}{r^f_t} \le \alpha\norm{\Delta X_t}^2 + \frac{\varepsilon_f^2}{\alpha}.
  \end{align}
  The cross term $\ip{\Delta X_t}{\sigma\Delta Z_t}$ can be separated as
  \begin{align}
    2\ip{\Delta X_t}{\sigma\Delta Z_t} \le \alpha\norm{\Delta X_t}^2 + \frac{\sigma^2}{\alpha}\norm{\Delta Z_t}^2.
  \end{align}
  Substituting these into \eqref{eq:dX-squared} and taking expectations yields \eqref{eq:split-JP-DeltaX}.

  Next, for the backward difference $\Delta Y_t\coloneqq Y_t-\tilde Y_t$, let
  \begin{align}
    \delta F_t & \coloneqq F(X_t,t,\rho_t) - F(\tilde X_t,t,\rho_t),        \\
    r^F_t      & \coloneqq F(\tilde X_t,t,\rho_t) - \tilde F(\tilde X_t,t),
  \end{align}
  then we can express
  \begin{equation}
    \di\Delta Y_t = \left(\tfrac12(\norm{Z_t}^2-\norm{\tilde Z_t}^2) + \delta F_t + r^F_t\right)\dt + \Delta Z_t^\top \di W_t.
  \end{equation}
  Applying It\^{o}'s formula to $|\Delta Y_t|^2$ gives
  \begin{align}\label{eq:DeltaY-diff}
    \di |\Delta Y_t|^2 = 2\Delta Y_t\left(\tfrac12(\norm{Z_t}^2-\norm{\tilde Z_t}^2)+\delta F_t+r^F_t\right)\dt + \norm{\Delta Z_t}^2\dt + 2\Delta Y_t \Delta Z_t^\top \di W_t.
  \end{align}
  The last martingale term vanishes when taking expectations. For the first term, using $\norm{Z_t}^2-\norm{\tilde Z_t}^2 = \Delta Z_t^\top(Z_t+\tilde Z_t)$
  and Young's inequality $|ab|\le \beta a^2/4 + \beta^{-1} b^2$ with $a=\Delta Y_t$ and $b=\Delta Z_t^\top(Z_t+\tilde Z_t)$, we obtain
  \begin{align}
    \left|\Delta Y_t \Delta Z_t^\top (Z_t+\tilde Z_t)\right|
    \le \frac{\beta}{4}|\Delta Y_t|^2 + \frac{1}{\beta}\|\Delta Z_t\|^2\|Z_t+\tilde Z_t\|^2.
  \end{align}
  From Assumption~\ref{ass:grad-bound-global}, we have $\|Z_t\|\le \sigma C_Z$ and $\|\tilde Z_t\|\le \sigma C_Z$, so
  \(\|Z_t+\tilde Z_t\|^2\le (\|Z_t\|+\|\tilde Z_t\|)^2\le 4\sigma^2 C_Z^2\) a.s.; taking expectations yields
  \begin{align}
    \E\big[\Delta Y_t \Delta Z_t^\top (Z_t+\tilde Z_t)\big]
    \le \frac{\beta}{4}\E|\Delta Y_t|^2 + \frac{4\sigma^2 C_Z^2}{\beta}\E\norm{\Delta Z_t}^2.
  \end{align}
  For $\delta F_t$, by Assumption \ref{ass:regularity-0},
  \begin{align}
    \E|\delta F_t|^2 \le (L_F)^2\E\norm{\Delta X_t}^2.
  \end{align}
  By Young's inequality,
  \begin{align}
    2\E[\Delta Y_t \delta F_t]
    \le \beta\E|\Delta Y_t|^2 + \frac{1}{\beta}\E|\delta F_t|^2
    \le \beta\E|\Delta Y_t|^2 + \frac{(L_F)^2}{\beta}\E\norm{\Delta X_t}^2.
  \end{align}
  Similarly, for $r^F_t$,
  \begin{align}
    2\E[\Delta Y_t r^F_t] \le \beta\E|\Delta Y_t|^2 + \frac{1}{\beta}\varepsilon_F^2.
  \end{align}
  Combining these with \eqref{eq:DeltaY-diff} yields \eqref{eq:split-JP-DeltaY}.
\end{proof}

\begin{lemma}[Forward $Z$-energy estimate]\label{thm:Z_energy_estimate}
  Assume that Assumptions~\ref{ass:global-consistency}, \ref{ass:global-consistency-learned}, \ref{ass:regularity-0}, and \ref{ass:grad-bound-global} hold.
  Then, the error $\int_0^T \|\Delta Z_t\|^2 \di t$ satisfies
  \begin{align}
    \begin{aligned}
       & \int_0^T \E \|\Delta Z_t\|^2\,\di t \\
       & \le
      \frac{1}{1-\eta/2}
      \Bigg[
        \E|\Delta Y_T|^2 + \E|\Delta Y_0|^2
        + C_{ZX}\, \E\int_0^T \|\Delta X_t\|^2\,\di t
        + C_{ZY}\, \E\int_0^T |\Delta Y_t|^2\,\di t
        + C_{Z,\varepsilon_F}\, T \varepsilon_F^2
        \Bigg].
    \end{aligned}\label{eq:Z-est-general}
  \end{align}
\end{lemma}

\begin{proof}[Proof of Lemma~\ref{thm:Z_energy_estimate}]
  Integrating \eqref{eq:DeltaY-diff} over $[0,T]$ and taking expectations, the martingale term
  vanishes, giving
  \begin{align}
     & \E|\Delta Y_T|^2 - \E|\Delta Y_0|^2         \\
     & =
    \E\int_0^T
    2\Delta Y_t\Bigl(\tfrac12(\norm{Z_t}^2-\norm{\tilde Z_t}^2)+\delta F_t+r^F_t\Bigr)\,dt
    + \E\int_0^T \norm{\Delta Z_t}^2\,dt \nonumber \\
     & =
    \E\int_0^T \norm{\Delta Z_t}^2\,dt
    + \E\int_0^T \Delta Y_t\bigl(\norm{Z_t}^2-\norm{\tilde Z_t}^2\bigr)\,dt
    + 2\E\int_0^T \Delta Y_t(\delta F_t+r^F_t)\,dt.
    \label{eq:Energy_identity_forZ_correct}
  \end{align}
  Rearranging and using $-a\le |a|$ gives
  \begin{align}
     & \E\int_0^T \norm{\Delta Z_t}^2\,dt                    \\
     & =
    \E|\Delta Y_T|^2 - \E|\Delta Y_0|^2
    - \E\int_0^T \Delta Y_t\bigl(\norm{Z_t}^2-\norm{\tilde Z_t}^2\bigr)\,dt
    - 2\E\int_0^T \Delta Y_t(\delta F_t+r^F_t)\,dt \nonumber \\
     & \le
    \E|\Delta Y_T|^2 + \E|\Delta Y_0|^2
    + \E\int_0^T \bigl|\Delta Y_t\bigl(\norm{Z_t}^2-\norm{\tilde Z_t}^2\bigr)\bigr|\,dt
    + 2\E\int_0^T |\Delta Y_t|\,\bigl(|\delta F_t|+|r^F_t|\bigr)\,dt.
    \label{eq:Energy_ineq_forZ_start}
  \end{align}
  For any $\eta>0$, we bound the quadratic term. Using
  $\norm{Z_t}^2-\norm{\tilde Z_t}^2=(Z_t+\tilde Z_t)^\top\Delta Z_t$ and Cauchy--Schwarz inequality, we have
  \begin{align}
    \bigl|\Delta Y_t(\norm{Z_t}^2-\norm{\tilde Z_t}^2)\bigr|
     & =
    \bigl|\Delta Y_t (Z_t+\tilde Z_t)^\top \Delta Z_t\bigr|
    \le
    |\Delta Y_t|\,\norm{Z_t+\tilde Z_t}\,\norm{\Delta Z_t} \nonumber \\
     & \le
    \frac{\eta}{2}\norm{\Delta Z_t}^2
    + \frac{1}{2\eta}\norm{Z_t+\tilde Z_t}^2\,|\Delta Y_t|^2.
  \end{align}
  By Assumption~\ref{ass:grad-bound-global}, $\|Z_t\|\le \sigma C_Z$ and $\|\tilde Z_t\|\le \sigma C_Z$,
  so $\|Z_t+\tilde Z_t\|^2\le (2\sigma C_Z)^2=4\sigma^2 C_Z^2$. Therefore,
  \begin{align}\label{eq:Young_quadratic_bound_forZ}
    \bigl|\Delta Y_t(\norm{Z_t}^2-\norm{\tilde Z_t}^2)\bigr|
    \le
    \frac{\eta}{2}\norm{\Delta Z_t}^2
    + \frac{2\sigma^2 C_Z^2}{\eta}\,|\Delta Y_t|^2.
  \end{align}
  Next, for any $\gamma>0$, Young's inequality yields
  \begin{align}
    2|\Delta Y_t\,\delta F_t|
     & \le \frac{\gamma}{2}|\Delta Y_t|^2 + \frac{2}{\gamma}|\delta F_t|^2
    \le \frac{\gamma}{2}|\Delta Y_t|^2 + \frac{2(L_F)^2}{\gamma}\norm{\Delta X_t}^2,
    \label{eq:Young_deltaF_forZ}
  \end{align}
  where we used $|\delta F_t|\le L_F\|\Delta X_t\|$ from Assumption~\ref{ass:regularity-0}. Similarly,
  \begin{align}
    2|\Delta Y_t\,r^F_t|
    \le \frac{\gamma}{2}|\Delta Y_t|^2 + \frac{2}{\gamma}|r^F_t|^2
    \le \frac{\gamma}{2}|\Delta Y_t|^2 + \frac{2}{\gamma}\varepsilon_F^2.
    \label{eq:rF_bound_forZ}
  \end{align}
  Substituting \eqref{eq:Young_quadratic_bound_forZ}, \eqref{eq:Young_deltaF_forZ}, and \eqref{eq:rF_bound_forZ}
  into \eqref{eq:Energy_ineq_forZ_start}, and integrating over $t\in[0,T]$, we obtain
  \begin{align}
    \E\int_0^T \norm{\Delta Z_t}^2\,dt
     & \le
    \E|\Delta Y_T|^2 + \E|\Delta Y_0|^2
    + \frac{\eta}{2}\E\int_0^T \norm{\Delta Z_t}^2\,dt \nonumber \\
     & \quad
    + \left(\frac{2\sigma^2 C_Z^2}{\eta}+\gamma\right)\E\int_0^T |\Delta Y_t|^2\,dt
    + \frac{2(L_F)^2}{\gamma}\E\int_0^T \norm{\Delta X_t}^2\,dt
    + \frac{2T}{\gamma}\varepsilon_F^2.
    \label{eq:Z_energy_before_rearrange}
  \end{align}
  Rearranging and assuming $\eta\in(0,2)$ yields
  the claimed estimate.
\end{proof}

\begin{lemma}[Forward total energy bound]\label{thm:explicit_total_error_bound}
  Assume that Assumptions~\ref{ass:global-consistency}, \ref{ass:global-consistency-learned}, \ref{ass:regularity-0}, and \ref{ass:grad-bound-global} hold.
  Assume also \eqref{eq:Ddet_def}--\eqref{eq:diag_pos_def}.
  Then, with $\mathcal C_E$ defined in \eqref{eq:CE_def},
  \begin{align}\label{eq:E_coarse_bound}
    \mathcal E^{\mathrm{fwd}}
    \le
    \mathcal C_E\, (\E|\Delta Y_T|^2+\E|\Delta Y_0|^2 + T\varepsilon_f^2 + T\varepsilon_F^2).
  \end{align}
\end{lemma}

\begin{proof}[Proof of Lemma~\ref{thm:explicit_total_error_bound}]
  We define
  \begin{align}
    \mathcal X(t)\coloneqq \E\|\Delta X_t\|^2,\qquad
    \mathcal Y(t)\coloneqq \E|\Delta Y_t|^2,\qquad
    \mathcal Z\coloneqq \int_0^T\E\|\Delta Z_t\|^2\,dt,
  \end{align}
  \begin{align}
    \mathcal X_* \coloneqq \sup_{t\in[0,T]}\mathcal X(t),\qquad
    \mathcal Y_* \coloneqq \sup_{t\in[0,T]}\mathcal Y(t),\qquad
    \mathcal E^{\mathrm{fwd}} = \mathcal X_*+\mathcal Y_*+\mathcal Z.
  \end{align}
  First, we bound $\mathcal X_*$.
  From \eqref{eq:split-JP-DeltaX}, Gr\"onwall's inequality gives, for all $t\in[0,T]$,
  \begin{align}
    \mathcal X(t)
    \le
    e^{C_{XX}t}\mathcal X(0)
    + e^{C_{XX}t} \int_0^t e^{-C_{XX}s}
    \Big(C_{XZ}\E\|\Delta Z_s\|^2 + C_{X,\varepsilon_f}\varepsilon_f^2\Big)\,\di s.
  \end{align}
  Since $e^{C_{XX}t}\le E_X$, $\mathcal X(0)=0$ (because $\Delta X_0=0$), and
  $\int_0^t \E\|\Delta Z_s\|^2ds\le \mathcal Z$, we obtain
  \begin{align}\label{eq:Xstar_pre}
    \mathcal X_* \le E_X\left(C_{XZ}\mathcal Z + C_{X,\varepsilon_f}T\varepsilon_f^2\right).
  \end{align}
  Next, we bound $\mathcal Y_*$.
  From \eqref{eq:split-JP-DeltaY}, Gr\"onwall's inequality gives, for all $t\in[0,T]$,
  \begin{align}
    \mathcal Y(t)
    \le
    e^{C_{YY}t}\mathcal Y(0)
    +
    e^{C_{YY}t}\int_0^t e^{-C_{YY}s}
    \Big(C_{YX}\mathcal X(s)+C_{YZ}\E\|\Delta Z_s\|^2 + C_{Y,\varepsilon_F}\varepsilon_F^2\Big)\,\di s.
  \end{align}
  Since $e^{C_{YY}t}\le E_Y$, $\int_0^t \mathcal X(s)ds\le T\mathcal X_*$, and
  $\int_0^t \E\|\Delta Z_s\|^2ds\le \mathcal Z$, we obtain
  \begin{align}\label{eq:Ystar_pre}
    \mathcal Y_*
    \le
    E_Y\Big(\mathcal Y(0)+C_{YX}T\mathcal X_* + C_{YZ}\mathcal Z + C_{Y,\varepsilon_F}T\varepsilon_F^2\Big).
  \end{align}
  Then we bound $\mathcal Z$.
  By Lemma~\ref{thm:Z_energy_estimate},
  \begin{align}
    \mathcal Z
    \le
    \kappa\Big(
    \mathcal K + C_{ZX}\int_0^T \mathcal X(t)\,dt + C_{ZY}\int_0^T \mathcal Y(t)\,dt
    \Big),
  \end{align}
  where
  \begin{align}
    \mathcal K \coloneqq \E|\Delta Y_T|^2 + \E|\Delta Y_0|^2 + C_{Z,\varepsilon_F}\,T\,\varepsilon_F^2.
  \end{align}
  Since $\int_0^T\mathcal X(t)dt\le T\mathcal X_*$ and $\int_0^T\mathcal Y(t)dt\le T\mathcal Y_*$,
  we obtain
  \begin{align}\label{eq:Z_pre}
    \mathcal Z \le \kappa\Big(\mathcal K + C_{ZX}T\mathcal X_* + C_{ZY}T\mathcal Y_*\Big).
  \end{align}
  We now close the inequalities for $\mathcal X_*$ and $\mathcal Y_*$.
  We define
  \begin{align}
    \begin{aligned}
      a_0 & \coloneqq E_X\Big(C_{X,\varepsilon_f}\,T\,\varepsilon_f^2 + C_{XZ}\,\kappa\,\mathcal K\Big),                 \\
      b_0 & \coloneqq E_Y\Big(\mathcal Y(0) + C_{Y,\varepsilon_F}\,T\,\varepsilon_F^2 + C_{YZ}\,\kappa\,\mathcal K\Big).
    \end{aligned}\label{eq:inhom_def}
  \end{align}
  Substituting \eqref{eq:Z_pre} into \eqref{eq:Xstar_pre} yields
  \begin{align}
    \mathcal X_*
    \le
    E_X\Big(C_{XZ}\kappa(\mathcal K + C_{ZX}T\mathcal X_* + C_{ZY}T\mathcal Y_*)
    + C_{X,\varepsilon_f}T\varepsilon_f^2\Big).
  \end{align}
  Rearranging gives
  \begin{align}\label{eq:ineq_Xstar_closed}
    (1-a_X)\mathcal X_* - a_Y\mathcal Y_* \le a_0,
  \end{align}
  with $a_X,a_Y,a_0$ as defined in \eqref{eq:coeff_a_b_def} and \eqref{eq:inhom_def}.
  Similarly, substituting \eqref{eq:Z_pre} into \eqref{eq:Ystar_pre} yields
  \begin{align}
    \mathcal Y_*
    \le
    E_Y\Big(\mathcal Y(0)+C_{YX}T\mathcal X_*
    + C_{YZ}\kappa(\mathcal K + C_{ZX}T\mathcal X_* + C_{ZY}T\mathcal Y_*)
    + C_{Y,\varepsilon_F}T\varepsilon_F^2\Big),
  \end{align}
  which rearranges to
  \begin{align}\label{eq:ineq_Ystar_closed}
    -b_X\mathcal X_* + (1-b_Y)\mathcal Y_* \le b_0,
  \end{align}
  with $b_X,b_Y,b_0$ as defined in \eqref{eq:coeff_a_b_def} and \eqref{eq:inhom_def}.
  We solve the resulting linear system.
  Since $a_Y,b_X\ge 0$, $1-a_X>0$, $1-b_Y>0$ by \eqref{eq:diag_pos_def},
  and $D_{\mathrm{det}}=(1-a_X)(1-b_Y)-a_Yb_X>0$ by assumption \eqref{eq:Ddet_def},
  multiplying \eqref{eq:ineq_Xstar_closed} by $(1-b_Y)$ and
  \eqref{eq:ineq_Ystar_closed} by $a_Y$ preserves the inequality directions and yields
  \begin{align}
    \mathcal X_* \le \frac{(1-b_Y)a_0 + a_Y b_0}{D_{\mathrm{det}}}.
  \end{align}
  Likewise, multiplying \eqref{eq:ineq_Xstar_closed} by $b_X$ and
  \eqref{eq:ineq_Ystar_closed} by $(1-a_X)$ gives
  \begin{align}
    \mathcal Y_* \le \frac{b_X a_0 + (1-a_X)b_0}{D_{\mathrm{det}}}.
  \end{align}
  Finally, we bound $\mathcal E^{\mathrm{fwd}}$.
  Substituting the above bounds into \eqref{eq:Z_pre} gives
  \begin{align}
    \begin{aligned}\label{eq:E_explicit}
      \mathcal E^{\mathrm{fwd}}
       & =\mathcal X_*+\mathcal Y_*+\mathcal Z                                                            \\
       & \le
      \Big(1+\kappa C_{ZX}T\Big)\mathcal X_* + \Big(1+\kappa C_{ZY}T\Big)\mathcal Y_* + \kappa\mathcal K, \\
       & \le
      \Big(1+\kappa C_{ZX}T\Big)\frac{(1-b_Y)a_0 + a_Y b_0}{D_{\mathrm{det}}}
      + \Big(1+\kappa C_{ZY}T\Big)\frac{b_X a_0 + (1-a_X)b_0}{D_{\mathrm{det}}}
      + \kappa\mathcal K.
    \end{aligned}
  \end{align}
  In particular, define the shorthand notations
  \begin{align}
    S_{\mathrm{err}} & \;\coloneqq\; \E|\Delta Y_T|^2+\E|\Delta Y_0|^2 + T\varepsilon_f^2 + T\varepsilon_F^2,\label{eq:Serr_def} \\
    C_K              & \;\coloneqq\; 1 + C_{Z,\varepsilon_F},\label{eq:CK_def}
  \end{align}
  so that $\mathcal K \le C_K\,S_{\mathrm{err}}$.
  Then, we have
  \begin{align}
    a_0
     & = E_X\Big(C_{X,\varepsilon_f}T\varepsilon_f^2 + C_{XZ}\kappa\mathcal K\Big)
    \le E_X\Big(C_{X,\varepsilon_f} + C_{XZ}\kappa C_K\Big)\,S_{\mathrm{err}},\label{eq:a0_bound}  \\
    b_0
     & = E_Y\Big(\mathcal Y(0) + C_{Y,\varepsilon_F}T\varepsilon_F^2 + C_{YZ}\kappa\mathcal K\Big)
    \le E_Y\Big(1 + C_{Y,\varepsilon_F} + C_{YZ}\kappa C_K\Big)\,S_{\mathrm{err}}.\label{eq:b0_bound}
  \end{align}
  Consequently, from \eqref{eq:E_explicit} and the definitions \eqref{eq:CE_def},
  \begin{align}
    \mathcal E^{\mathrm{fwd}}
    \le
    \mathcal C_E\, S_{\mathrm{err}},
  \end{align}
  which is exactly \eqref{eq:E_coarse_bound}.
\end{proof}

\begin{lemma}[Backward total energy bound]\label{thm:explicit_total_error_bound_backward_compact}
  Assume that Assumptions~\ref{ass:global-consistency}, \ref{ass:global-consistency-learned}, \ref{ass:regularity-0}, and \ref{ass:grad-bound-global} hold.
  Assume also \eqref{eq:Ddet_def}--\eqref{eq:diag_pos_def}.
  Then, with the same constant $\mathcal C_E$,
  \begin{align}\label{eq:backward_total_energy_bound_compact}
    \mathcal E^{\mathrm{bwd}}
    \le
    \mathcal C_E\Big(
    \E|\Delta\widehat Y_0|^2+\E|\Delta\widehat Y_T|^2
    +T\varepsilon_f^2+T\varepsilon_F^2
    \Big).
  \end{align}
\end{lemma}

\begin{proof}[Proof of Lemma~\ref{thm:explicit_total_error_bound_backward_compact}]
  We rewrite the backward-time system as a forward problem by time reversal.
  For $t\in[0,T]$, define
  \begin{align}
    \check X_t \coloneqq \widehat X_{T-t},
    \qquad
    \tilde{\check X}_t \coloneqq \tilde{\widehat X}_{T-t},
    \qquad
    \check Y_t \coloneqq -\widehat Y_{T-t},
    \qquad
    \tilde{\check Y}_t \coloneqq -\tilde{\widehat Y}_{T-t},
  \end{align}
  and
  \begin{align}
    \check Z_t \coloneqq -\widehat Z_{T-t},
    \qquad
    \tilde{\check Z}_t \coloneqq -\tilde{\widehat Z}_{T-t}.
  \end{align}
  Also let
  \begin{align}
    \check W_t \coloneqq W_T-W_{T-t},
  \end{align}
  which is a Brownian motion with respect to the reversed filtration.
  A direct substitution shows that $(\check X,\check Y,\check Z)$ and
  $(\tilde{\check X},\tilde{\check Y},\tilde{\check Z})$ satisfy forward-time FBSDEs of the same algebraic form as \eqref{eq:fbsde-forward-full-x}--\eqref{eq:fbsde-forward-full-y},
  with the same constants $L_f,L_F,C_Z$ and with the coupled initial condition
  \begin{align}
    \tilde{\check X}_0=\check X_0
    \qquad\text{a.s.},
  \end{align}
  because $\tilde{\widehat X}_T=\widehat X_T$ a.s.\ by Assumption~\ref{ass:global-consistency-learned}.
  Therefore, the proof of Lemma~\ref{thm:explicit_total_error_bound} applies to the reversed pair.
  Since
  \begin{align}
    \sup_{t\in[0,T]}\E\|\check X_t-\tilde{\check X}_t\|^2
     & = \sup_{s\in[0,T]}\E\|\Delta\widehat X_s\|^2, \\
    \sup_{t\in[0,T]}\E|\check Y_t-\tilde{\check Y}_t|^2
     & = \sup_{s\in[0,T]}\E|\Delta\widehat Y_s|^2,   \\
    \int_0^T \E\|\check Z_t-\tilde{\check Z}_t\|^2\,dt
     & = \int_0^T \E\|\Delta\widehat Z_s\|^2\,ds,
  \end{align}
  the left-hand side is exactly $\mathcal E^{\mathrm{bwd}}$.
  The boundary terms of the reversed system become
  \begin{align}
    \E|\check Y_T-\tilde{\check Y}_T|^2 = \E|\Delta\widehat Y_0|^2,
    \qquad
    \E|\check Y_0-\tilde{\check Y}_0|^2 = \E|\Delta\widehat Y_T|^2,
  \end{align}
  which proves \eqref{eq:backward_total_energy_bound_compact}.
\end{proof}

\begin{proof}[Proof of Theorem~\ref{thm:combined_total_energy_bound}]
  By Lemma~\ref{thm:explicit_total_error_bound}, there exists a constant
  $\mathcal C_E>0$ such that
  \begin{align}\label{eq:fwd_energy_recall}
    \mathcal E^{\mathrm{fwd}}
    \le
    \mathcal C_E\Big(
    \E|\Delta Y_0|^2+\E|\Delta Y_T|^2
    +T\varepsilon_f^2+T\varepsilon_F^2
    \Big),
  \end{align}
  Likewise, by Lemma~\ref{thm:explicit_total_error_bound_backward_compact},
  the same constant $\mathcal C_E>0$ satisfies
  \begin{align}\label{eq:bwd_energy_recall}
    \mathcal E^{\mathrm{bwd}}
    \le
    \mathcal C_E\Big(
    \E|\Delta \widehat Y_0|^2+\E|\Delta \widehat Y_T|^2
    +T\varepsilon_f^2+T\varepsilon_F^2
    \Big),
  \end{align}
  Adding \eqref{eq:fwd_energy_recall} and \eqref{eq:bwd_energy_recall} yields
  \begin{align}\label{eq:sum_energy_pre}
    \mathcal E^{\mathrm{fwd}}+\mathcal E^{\mathrm{bwd}}
    \le
    \mathcal C_E\Big(
    \E|\Delta Y_0|^2+\E|\Delta\widehat Y_0|^2
    +\E|\Delta Y_T|^2+\E|\Delta\widehat Y_T|^2
    +2T\varepsilon_f^2+2T\varepsilon_F^2
    \Big).
  \end{align}
  The endpoint representative in
  Assumptions~\ref{ass:global-consistency}(iv) and
  \ref{ass:global-consistency-learned}(iv), together with the endpoint
  couplings in Assumption~\ref{ass:global-consistency-learned}(i)--(ii) and the
  marginal consistency in Assumptions~\ref{ass:global-consistency}(iii) and
  \ref{ass:global-consistency-learned}(iii), gives
  \[
    \Delta Y_0=0,
    \qquad
    \Delta Y_T=0,
    \qquad
    \Delta\widehat Y_0=0,
    \qquad
    \Delta\widehat Y_T=0
    \quad\text{a.s.}
  \]
  Substituting these identities into \eqref{eq:sum_energy_pre} gives
  \eqref{eq:combined_energy_bound}.
\end{proof}

\section{Computational Environment}\label{app:computational_environment}

The numerical experiments were run on two GPU environments, summarized in
Table~\ref{tab:computational_environment}.
Code will be released upon publication.

\begin{table}[ht]
  \centering
  \caption{Computational environments used for the numerical experiments.}
  \begin{tabular}{p{0.35\linewidth}p{0.55\linewidth}}
    \hline
    Experiments & GPU environment                                \\
    \hline
    GMM dynamics
                & NVIDIA L40S PCIe GPU with 48 GB memory.        \\
    V-neck and Opinion dynamics
                & NVIDIA GeForce RTX 3090 GPU with 24 GB memory. \\
    \hline
  \end{tabular}
  \label{tab:computational_environment}
\end{table}


\end{document}